\documentclass[11pt]{amsart}
\title[Functionals representing Sobolev norms]{Families of functionals representing  Sobolev norms}
\author[H. Brezis  \ \  \  A. Seeger \ \ \ J. Van Schaftingen \ \ \ P.-L. Yung]
{Ha\"\i m Brezis \ \ \ Andreas Seeger \ \ \ Jean Van Schaftingen \ \ \ Po-Lam Yung}

\address{Department of Mathematics\\
  Rutgers University, Hill Center, Busch Campus\\ 
  110 Frelinghuysen Road, Piscataway, NJ 08854, USA}
  \address{
  Departments of Mathematics and Computer Science\\ Technion, Israel Institute of Technology\\ 32.000 Haifa, Israel}
\address{
  Laboratoire Jacques-Louis Lions\\
  Sorbonne Universit\'es, UPMC Universit\'e Paris-6, 4  place Jussieu\\
  75005 Paris, France}
\email{brezis@math.rutgers.edu}

\address{Department of Mathematics \\ 
University of Wisconsin, Madison \\
480 Lincoln Drive, Madison, WI, 53706, USA}
\email{seeger@math.wisc.edu}

\address{Universit\'e catholique de Louvain\\ 
Institut de Recherche en Math\'ematique et Physique\\
Chemin du Cyclotron 2 bte L7.01.01\\
1348 Louvain-la-Neuve\\
Belgium}
\email{Jean.VanSchaftingen@UCLouvain.be}

\address{Mathematical Sciences Institute \\
Australian National University \\
Canberra ACT 2601 \\
Australia} 
\email{PoLam.Yung@anu.edu.au}

\usepackage{enumerate, fullpage}

\usepackage[T1]{fontenc}
\usepackage[utf8]{inputenc}
\usepackage{geometry}

\usepackage{lmodern} 
\usepackage{microtype}

\usepackage{amssymb}
\usepackage{amsmath}
\usepackage{amsthm}
\usepackage{esint}
\usepackage{mathtools}
\usepackage[mathscr]{eucal}
\usepackage{constants}
\usepackage{paralist}
\usepackage{mathrsfs}
\newcommand{\defeq}{\coloneqq}

\usepackage[nameinlink%
,capitalize%
]{cleveref}

\usepackage[backrefs,abbrev]{amsrefs}

\newtheorem{theorem}{Theorem}[section]
\newtheorem{proposition}[theorem]{Proposition}
\newtheorem{lemma}[theorem]{Lemma}
\newtheorem{corollary}[theorem]{Corollary}
\newtheorem{examplebf}[theorem]{Example}

\theoremstyle{definition}

\theoremstyle{remark}
\newtheorem{remark}[theorem]{Remark}
\newtheorem*{remark*}{Remark}
\newtheorem*{remarks*}{Remarks}

\numberwithin{equation}{section}

\usepackage{mathtools}\DeclarePairedDelimiter{\abs}{\lvert}{\rvert}
\DeclarePairedDelimiter{\norm}{\lVert}{\rVert}
\DeclarePairedDelimiter{\quasinorm}{[}{]}

\DeclarePairedDelimiter{\brk}{(}{)}
\DeclarePairedDelimiter{\set}{\{}{\}}

\DeclarePairedDelimiterX\dualprod[2]{\langle}{\rangle}{#1, #2}
\DeclarePairedDelimiterX\intvo[2]{(}{)}{#1, #2}
\DeclarePairedDelimiterX\intvc[2]{[}{]}{#1, #2}
\DeclarePairedDelimiterX\intvl[2]{(}{]}{#1, #2}
\DeclarePairedDelimiterX\intvr[2]{[}{)}{#1, #2}

\newcommand{\st}{\;:\;}

\newcommand{\Rset}{\mathbb{R}}
\newcommand{\Nset}{\mathbb{N}}

\newcommand{\Sset}{\mathbb{S}}

\newcommand{\dif}{\,\mathrm{d}}

 \makeatletter
\renewcommand\subsection{\@startsection{subsection}{2}%
   \z@{.5\linespacing\@plus.7\linespacing}{.1\linespacing}%
   {\normalfont\itshape}} 
\makeatother
\newcommand{\R}{\ensuremath{\mathbb{R}}}
\renewcommand{\S}{\ensuremath{\mathbb{S}}}

\newcommand{\eps}{\ensuremath{\varepsilon}}

\usepackage{bbm}
\def\bbone{{\mathbbm 1}}
\def\om{\omega}

\newcommand{\Be}{\begin{equation}}
\newcommand{\Ee}{\end{equation}}
\def\lc{\lesssim}

\def\la{{\lambda}}
\def\ga{{\gamma}}
\def\ka{{\kappa}}

\def\cf{{\it cf}}

\def\loc{{\text{\rm loc}}}

\def\supp{{\text{\rm supp}}}

\def\inn#1#2{\langle#1,#2\rangle}

\def\Biginn#1#2{\Bigl\langle#1,#2\Bigr\rangle}

\def\fM{{\mathfrak {M}}}

\def\fS{{\mathfrak {S}}}

\def\bbN{{\mathbb {N}}}

\def\bbR{{\mathbb {R}}}
\def\bbS{{\mathbb {S}}}

\def\bbZ{{\mathbb {Z}}}
\def\cA{{\mathcal {A}}}

\def\cE{{\mathcal {E}}}

\def\cL{{\mathcal {L}}}
\def\cM{{\mathcal {M}}}
\def\cN{{\mathcal {N}}}

\def\cR{{\mathcal {R}}}

\def\cV{{\mathcal {V}}}
\def\cW{{\mathcal {W}}}

\begin{document}
\begin{abstract}
We obtain new characterizations of the Sobolev spaces $\dot W^{1,p}(\mathbb R^N)$ and the bounded variation space $\dot {BV}(\mathbb R^N)$. The characterizations are in terms of the functionals $\nu_{\ga} (E_{\lambda,\gamma/p}[u])$ where 
\[
E_{\lambda,\gamma/p}[u]= \Big\{(x,y )\in \mathbb R^N \times \mathbb R^N \colon x \neq y, \, \frac{|u(x)-u(y)|}{|x-y|^{1+\gamma/p}}>\lambda\Big\}
\]
and the measure $\nu_\gamma$ is given by $\mathrm{d} \nu_\gamma(x,y)=|x-y|^{\gamma-N} \mathrm{d} x \mathrm{d} y$. We provide characterizations  which involve  the $L^{p,\infty}$-quasi-norms $\sup_{\lambda>0} \la\, \nu_{\gamma} (E_{\lambda,\gamma/p}[u]) ^{1/p}$  and also exact formulas via   corresponding limit functionals, with  the limit for   $\lambda\to\infty$ when $\gamma>0$ and  the limit for  $\lambda\to 0^+$ when  $\gamma<0$. The  results unify and substantially extend previous work  by Nguyen and by Brezis, Van Schaftingen and Yung. For $p>1$ the characterizations hold for all $\gamma\neq 0$. For  $p=1$ the upper bounds for the $L^{1,\infty}$ quasi-norms  fail in the range $\gamma\in [-1,0) $; moreover in this case the limit functionals represent the $L^1$ norm of the gradient   for $C^\infty_c$-functions but not for generic $\dot W^{1,1}$-functions. For this situation   we provide new counterexamples  which are built on self-similar sets of dimension $\gamma+1$.  For $\gamma=0$ the  characterizations of Sobolev spaces fail;  however we  obtain a new formula for the Lipschitz norm via the expressions $\nu_0(E_{\lambda,0}[u])$.
\end{abstract}
\maketitle

\section{Introduction}
In this article, we are  concerned with various ways in which we can recover the Sobolev semi-norm $\|\nabla u\|_{L^p(\bbR^N)}$ via positive non-convex  functionals involving  differences $u(x)-u(y)$.

We begin by mentioning two relevant results already in the literature. A theorem of   H.-M. Nguyen \cite{Nguyen06} (see also \citelist{\cite{BN2018}\cite{BN2020}}) states that for  $1<p<\infty$ and $u$ in the inhomogeneous Sobolev space $W^{1,p}(\bbR^N)$, 
\begin{equation} \label{eq:Nguyen}
\lim_{\lambda \searrow 0} \lambda^p \iint_{|u(x)-u(y)| > \la} |x - y|^{-p-N} \dif x  \dif y = \frac{\ka(p,N)}{p} \norm{\nabla u}_{L^p(\R^N)}^p
\end{equation}
with
\begin{equation} \label{eq:kpN}
\ka (p,N) \defeq \int_{\Sset^{N-1}} |e \cdot \omega|^p \dif \omega
=\frac{2\Gamma(\tfrac{p+1}{2})\pi^{\frac{N-1}{2}}}{\Gamma(\tfrac{N+p}{2})},
\end{equation}
and $e$ is any unit vector in $\R^N$. As shown in \cite{BN2018},  \eqref{eq:Nguyen} still holds for all $u \in C^1_c(\R^N)$ when $p=1$ but fails for general $u\in  W^{1,1}(\R^N)$. The limit formula \eqref{eq:Nguyen}  may be compared to a  theorem  of three of the authors  \cite{Brezis_VanSchaftingen_Yung_2021}, which states  that for all $u \in C^{\infty}_c(\R^N)$ and $1 \leq p < \infty$, one has
\begin{equation} \label{eq:BVYlimit}
\lim_{\lambda \to \infty} \lambda^p \mathcal{L}^{2 N}\big (\{(x,y) \in \R^N \times \R^N \colon |u(x)-u(y)| > \la |x-y|^{1+\frac Np} \}\big) = \frac{\ka(p,N)}{N} \norm{\nabla u}_{L^p(\R^N)}^p
\end{equation}
where $\mathcal{L}^{2 N}$ denotes the Lebesgue measure on $\R^N \times \R^N$.
Our first result, namely \cref{theorem_weighted_reverse} below, provides an extension of \eqref{eq:Nguyen} and \eqref{eq:BVYlimit} that unifies the two statements. Before we state the theorem, we introduce some notations that will be used throughout the paper. 
 
First, for  Lebesgue measurable subsets $E$ of $\bbR^{2N} = \bbR^N \times \bbR^N$ and $\ga \in \R$,  we define 
\Be  \label{eq:nugamma-definition} 
\nu_\gamma(E)\defeq \iint
\limits_{\substack{(x,y) \in E\\ x\neq y}} 
|x-y|^{\gamma-N} \dif x \dif y.
\Ee
In particular, when $\ga = N$, $\nu_N$ is just the Lebesgue measure on $\R^{2N}$. If $u$ is a measurable function on $\R^N$ and $b \in \R$, we define, for $(x,y) \in \bbR^N \times \bbR^N$ with $x\neq y$, a difference quotient
\Be
\label{Qbetadef} Q_b u(x,y) \defeq  \frac{u(x)-u(y)}{|x-y|^{1+b}}; 
\Ee 
moreover, we define, for $\lambda > 0$, the superlevel set of $Q_bu$ at height $\la$ by
\begin{equation} \label{Elg_def}
E_{\lambda,b}[u] \defeq \Big\{(x,y) \in \R^N \times \R^N \colon x \ne y, \,| Q_{b} u(x,y)| > \lambda \Big\}.
\end{equation}
We will  denote by  $\dot W^{1,p}(\bbR^N)$, $p \ge 1$  the homogeneous Sobolev space, i.e.\ the space of $L^1_\loc(\bbR^N)$ functions for which the distributional gradient $\nabla u$ belongs to $L^p(\bbR^N)$, with the   semi-norm  $\|u\|_{\dot W^{1,p}}\defeq \norm{\nabla u}_{L^p(\bbR^N)}.$  The  inhomogeneous Sobolev space $W^{1,p}$ is the subspace of $\dot W^{1,p}$-functions $u$ for which $u\in L^p$, and we set $\|u\|_{W^{1,p}} \defeq \|u\|_{L^p}+\|\nabla u\|_{L^p}$. For $p=1$ we will also consider the space $\dot {BV}(\bbR^N)$ of functions of bounded variations, i.e.\ locally integrable functions $u$ for which the gradient $\nabla u\in \cM$ belongs to  the space $\cM$ of  $\bbR^N$-valued bounded Borel measures and we put $\|u\|_{\dot {BV}} \defeq \|\nabla u\|_\cM$; furthermore, let $BV \defeq \dot{BV} \cap L^1$. In the dual formulation, with $C^1_c$ denoting the space of $C^1$ functions with compact support,
\[
\|u\|_{\dot{BV} } \defeq \sup\Big\{ \Big| \int_{\R^N} u\, \mathrm{div}(\phi) \Big|: \phi\in C^1_c(\bbR^N,\bbR^N),\,\|\phi\|_\infty\le 1\Big\}.
\]
For general background material on Sobolev spaces see \cite{Brezis_2011}, \cite{stein-diff}.

\begin{theorem}
\label{theorem_weighted_reverse}
Suppose $N \geq 1$, $1 \leq p < \infty$, $\gamma\in \R \setminus \{0\}$. \begin{enumerate}[(a)]
\item If  $\gamma > 0$,  then  for all $u\in \dot W^{1,p}(\R^N)$ 
\begin{equation} \label{eq:lim_weighted}
\lim_{\lambda \to +\infty}   \lambda^p \nu_\gamma(E_{\la, \gamma/p}[u]) = \frac{\ka(p,N)}{|\ga|} \norm{\nabla u}_{L^p(\R^N)}^p.
\end{equation}
\item If either $\gamma < 0$, $p>1$ or $\gamma<-1$,  $p=1$  then for all $u\in  \dot W^{1,p}(\R^N) $
\begin{equation} \label{eq:lim_weighted_neg} 
\lim_{\lambda \searrow 0}  \lambda^p  \nu_\gamma(E_{\la, \gamma/p}[u]) = \frac{\ka(p,N)}{|\ga|} \norm{\nabla u}_{L^p(\R^N)}^p.
\end{equation}
\item \label{itemc}
 If $p=1$ and $-1\le \ga<0$  then   \eqref{eq:lim_weighted_neg}
remains true for all $u\in C^1_c(\R^N)$  but fails for generic $u\in \dot W^{1,1}(\R^N)$. However we still have for all $u\in \dot W^{1,1} (\bbR^N)$
\Be\label{eq:liminf} 
\liminf_{\la\searrow 0} \lambda \nu_\ga(E_{\la,\ga}[u]) \ge \frac{\kappa(1,N)}{|\gamma|}  \norm{\nabla u}_{L^1(\R^N)}.
\Ee
\end{enumerate}
\end{theorem}

Formula \eqref{eq:Nguyen} is the  special case of \eqref{eq:lim_weighted_neg} with $\gamma = -p$, and formula  \eqref{eq:BVYlimit} is the  special case of \eqref{eq:lim_weighted} with $\gamma = N$. Note that our result concerns functions in the homogeneous Sobolev space $\dot{W}^{1,p}$; we do not require $u$ to be in $L^p$.

\begin{remarks*} 
(i) The reader will note the resemblance of \eqref{eq:lim_weighted_neg} and \eqref{eq:lim_weighted} and  may wonder why  in \eqref{eq:lim_weighted_neg}, for $\ga<0$,  one is concerned with the limit as $\lambda \searrow 0$ and in \eqref{eq:lim_weighted}, for $\ga>0$,  one takes  the limit as  $\lambda \to \infty$. In the proofs of these formulas one relates limits involving $\la\nu_\gamma(E_{\la,\ga/p}[u])^{1/p} $ to (the absolute value of) limits of directional difference quotients $\delta^{-1}(u(x+\delta \theta)-u(x) ) $ with increment  $\delta=\la^{-p/\gamma}$,  and in order to recover the directional derivative $\inn{\theta}{\nabla u(x)}$ we need to let  $\delta\to 0$, which suggests that we need to take $\lambda \to \infty$ or $\lambda \searrow 0$ depending on the sign of $\gamma$. For the calculations  see the proofs  of \cref{lem:liminfs} and \cref{lem:limsupsN>1} below. 

\noindent (ii) The failure of  \eqref{eq:lim_weighted_neg} for $p=1$, $\ga\in [-1,0)$ and $u\in \dot W^{1,1}(\bbR^N)$ is generic in  the sense of Baire category. It may happen that $\lim_{\la \searrow 0}  \la\nu_\ga(E_{\la,\ga}[u])=\infty$. This phenomenon was originally revealed when $\ga=-1$ by A. Ponce and is presented in \cite{Nguyen06}, see also \cite[Pathology~1]{BN2018}.  For stronger statements and more information see \cref{thm:counterexamples}. For $\gamma\in (-1,0)$ we provide new examples based on self-similarity considerations. For discussion of failure in the case $\ga=0$ see \cref{thm:counterexamples-gamma=0} below. The special case of \eqref{eq:liminf} for $\gamma=-1$ was already established in \cite[Proposition 1]{BN2018}.
\end{remarks*}

When $p=1$ we can also consider what happens if one allows functions in $\dot{BV}(\bbR^N)$ in \eqref{eq:lim_weighted} and \eqref{eq:lim_weighted_neg}. For $\gamma=N$ in particular Poliakovsky \cite{poliakovsky} asked whether the limit formulas remain valid in this generality (with $\|\nabla u\|_{L^1}$ replaced by $\|\nabla u\|_\cM$). 
We provide a negative answer:
\begin{proposition} 
\label{prop:failureoflimitforBV} 
\begin{enumerate}[(i)]
\item The analogues of the limiting formulas  \eqref{eq:lim_weighted} for $\gamma>0$, $p=1$ and  \eqref{eq:lim_weighted_neg} for $\gamma<0$,  $p=1$,  with  $\|\nabla u\|_{\cM}$ on the right hand side, fail for suitable $u\in \dot{ BV}$.

\item Specifically, let $\Omega\subset \bbR^N$ be a bounded domain  with smooth boundary and let  $u$ be the characteristic function of $\Omega$. The limits $ \lim_{\lambda \to \infty}   \lambda \nu_\gamma(E_{\la, \gamma}[u]) $  for $\gamma>0$ and $ \lim_{\lambda \to 0+}   \lambda \nu_\gamma(E_{\la, \gamma}[u]) $ for $\gamma<-1$ exist, but they  are not equal to $|\gamma|^{-1}\kappa(1,N)\|\nabla u\|_\cM$.
\end{enumerate}
 \end{proposition}
For a more detailed discussion we refer to Section \ref{sec:BV-limit}. See also \cref{sec:Otherliminfs} for a discussion about some related open problems.

Motivated by \cite{Brezis_VanSchaftingen_Yung_2021}, we will also be interested in what happens to the larger quantity obtained by replacing the limits on the left hand side of \eqref{eq:lim_weighted} and \eqref{eq:lim_weighted_neg} by $\sup_{\lambda > 0} $. This will be formulated in terms of the Marcinkiewicz space   $L^{p,\infty}(\R^{2N},\nu_\ga) $ (a.k.a.\ weak type $L^{p}$)  defined by the condition 
\begin{equation}\label{eq:weaktype-p}
\quasinorm{F}_{L^{p, \infty} (\R^{2N},\nu_\ga)}^p \defeq \sup_{\lambda > 0} \lambda^p \nu_\ga \bigl(\{ (x,y) \in \bbR^N \times \bbR^N \st \abs{F (x,y)} > \lambda\}\bigr)< \infty.
\end{equation}
As an  immediate consequence of \cref{theorem_weighted_reverse} we have for $N\ge 1$, $1\le p<\infty$,  $\gamma\neq 0$ and all $u\in C^\infty_c(\R^N)$,
\begin{equation}\label{eq:Marc-lower}
\big[ Q_{ \ga/ p } u\big ]_{L^{p,\infty}(\R^{2N},\nu_\ga)}^p \ge C(N,p,\gamma) \norm{\nabla u}_{L^p(\R^N)}^p
\end{equation}
where $C(N,p,\gamma)$ is a positive constant depending only on $N$, $p$ and $\gamma$. Moreover, the same conclusion holds for all $u \in \dot{W}^{1,p}(\R^N)$ when $p > 1$ with any $\gamma \ne 0$, and when $p = 1$ with any $\gamma \notin [-1,0]$. We  shall show that the conditions in the last statement can in fact be relaxed, see the inequalities \eqref{eq:converse-p} and \eqref{eq:converse-1} below. In addition  we have the  important upper bounds for $Q_{\ga/p}u$, extending the case $\gamma = N$ already dealt with in \cite{Brezis_VanSchaftingen_Yung_2021} for $u \in C^{\infty}_c(\R^N)$. The result in \cite{Brezis_VanSchaftingen_Yung_2021} states that for every $N \geq 1$, there exists a constant $C(N)$ such that 
\begin{equation} \label{eq:new_char_sob_norm}
\big[ Q_{N/p } u\big ]_{L^{p,\infty}(\R^{2N},\nu_N)}^p
\leq C(N) \norm{\nabla u}_{L^p(\R^N)}^p
\end{equation}
for all $u \in C^{\infty}_c(\R^N)$ and all $1 \leq p < \infty$. 
In light of \cref{theorem_weighted_reverse}, it is natural to ask whether one can replace the limits on the left hand sides of \eqref{eq:lim_weighted} and \eqref{eq:lim_weighted_neg} by $\sup_{\lambda > 0}$ and still obtain a quantity that is comparable to $\norm{\nabla u}_{L^p(\R^N)}^p$. As suggested by \cref{theorem_weighted_reverse}   the  answer to our question is sensitive to the values of $\gamma$ and $p$.  
\begin{theorem} \label{theorem_weighted_pgr1}
Suppose that $N\ge 1$, $1<p<\infty$ and $\gamma\in \bbR$.  Then the following hold.
\begin{enumerate} [(i)] 
\item The inequality  
\begin{equation}\label{eq:Marc-p}
\big[ Q_{ \ga/ p } u\big ]_{L^{p,\infty}(\R^{2N},\nu_\ga)} \le C(N,p,\gamma) \norm{\nabla u}_{L^p(\R^N)}
\end{equation}
holds for all $u \in C^\infty_c(\R^N)$ if and only if $\gamma\neq 0$. In this case \eqref{eq:Marc-p} extends to all  $u\in \dot W^{1,p}(\bbR^N)$. \item Suppose that $u\in L^1_\loc(\bbR^N) $ and $Q_{\ga/p} u\in L^{p,\infty} (\R^{2N},\nu_\ga)$.  Then $u\in \dot W^{1,p}(\R^N) $ and we have the inequality
\Be \label{eq:converse-p}
\|\nabla u\|_{L^p(\R^N)} \le C_{N,p,\ga} [Q_{\ga/p} u]_{L^{p,\infty}(\R^{2N},\nu_\ga)} .
\Ee 
\end{enumerate}
\end{theorem} 

There is a new phenomenon for  $p=1$, namely  the upper bounds for $Q_\ga u$ only hold for the  more restrictive range $\gamma \in (-\infty,-1)\cup (0,\infty)$. Here it is also  natural to replace $\dot W^{1,1}$ with $\dot{BV}$.

\begin{theorem} \label{theorem_weighted_pequal1}
Suppose that $N\ge 1$ and $\gamma\in \bbR$.  Then the following hold.
\begin{enumerate}[(i)]
\item The inequality
  \begin{equation} \label{eq:Marc-1}
  \big[ Q_{\ga } u\big ]_{L^{1,\infty}(\R^{2N},\nu_\ga)} \le  C(N,\gamma) \|\nabla u\|_{L^1(\R^N)}
\end{equation}
holds for all $u\in C^\infty_c(\bbR^N)$ if and only if $ \gamma \not \in [-1, 0]$. In this case \eqref{eq:Marc-1} extends to all $u\in \dot W^{1,1}(\bbR^N)$, and, if $\|\nabla u\|_{L^1(\R^N)}$ is replaced by   $\|\nabla u\|_\cM$, to all $u\in \dot{BV}(\Rset^N)$.
\item Suppose that $u\in L^1_\loc(\bbR^N) $ and $Q_{\ga} u\in L^{1,\infty} (\bbR^{2N},\nu_\ga)$.  Then $u\in \dot {BV}(\R^N) $ and we have the inequality
\Be \label{eq:converse-1}
\|\nabla u\|_{\cM}  \le C_{N,\ga} [Q_{\ga} u]_{L^{1,\infty}(\R^{2N},\nu_\ga)} .
\Ee 
\end{enumerate}
\end{theorem} 

We note that the quantitative bounds \eqref{eq:Marc-p} and \eqref{eq:Marc-1} in \cref{theorem_weighted_pgr1,theorem_weighted_pequal1} are crucial tools for establishing the limiting relations for all $\dot W^{1,p}$ functions in \cref{theorem_weighted_reverse}. Note that there is no restriction on $\gamma$ in  \eqref{eq:converse-p} and \eqref{eq:converse-1}. The constants in the inequalities will be quantified further later in the paper.    In particular,  $C(N,p,\gamma)$ in \eqref{eq:Marc-p} remains bounded as $p\searrow 1$ only in the range  $\gamma \in (0,\infty)\cup(-\infty, -1)$  (cf.\  \cref{theorem_weighted_upper_Lp_allgamma} and \cref{proposition_Mp_gamma_-1_0_1}). 

\subsubsection*{Historical comments.}
Some  special cases of the above quantitative estimates have been known.
Estimate \eqref{eq:Marc-p} for $\gamma=-p$ and $1<p<\infty$ was discovered independently by H.M. Nguyen \cite{Nguyen06}, and by A. Ponce and J. Van Schaftingen (unpublished communication to H. Brezis and H.M. Nguyen), both relying on the Hardy-Littlewood maximal inequality. 
A. Poliakovsky \cite{poliakovsky} recently proved generalizations  of  results  in  \cite{Brezis_VanSchaftingen_Yung_2021} to Sobolev spaces on domains; moreover he obtained  \cref{theorem_weighted_pgr1,theorem_weighted_pequal1} in the special case $\gamma=N$ under the additional assumption that $u \in L^p$. 
 Other far-reaching generalizations to one-parameter families of operators were obtained by O. Dominguez and M. Milman \cite{dominguez-milman}.

\subsubsection*{The case $\gamma=0$}
We  shall now return to the necessity of the assumption $\gamma\notin[-1,0]$ in parts  of Theorems \ref{theorem_weighted_reverse},  \ref{theorem_weighted_pgr1} and \ref{theorem_weighted_pequal1}. When $\gamma=0$, the  bounds for $[Q_{\ga/p}u]_{L^{p,\infty}(\R^{2N},\nu_\ga)}$  fail in a striking way. We   begin by formulating a  result  illustrating  this failure, which also gives a characterization of the semi-norm in the Lipschitz space $\dot W^{1,\infty}$.

\begin{theorem} \label{thm:counterexamples-gamma=0} 
Suppose $N\ge 1$, $u$ is locally integrable on $\bbR^N$ and  $\nabla u\in L^1_\loc(\bbR^N)$.
Then 
\begin{equation} \label{eq:Lipschitz} 
\|\nabla u\|_{L^{\infty}(\R^N)} =\inf \{\la>0: \nu_0(E_{\la, 0}[u]) <\infty \} .
\end{equation}
\end{theorem}

 Indeed in \cref{thm:gamma=0} we shall prove the stronger statement  that $\nu_0(E_{\la,0}[u]) =0$ for $\la>\|\nabla u\|_\infty$, and $\nu_0(E_{\la,0}[u]) =\infty$ for $\la<\|\nabla u\|_\infty$.
As an  immediate consequence  of \cref{thm:counterexamples-gamma=0}  we get 
\begin{corollary}\label{cor:recognize-constants}
Let $u$ be  locally integrable on $\bbR^N$. If $\nabla u\in L^1_\loc(\bbR^N)$ and if $\nu_0(E_{\la, 0}[u])$ is finite for all   $\la>0$,  then $u$ is almost everywhere equal to a  constant function.
\end{corollary}

In view of other known results \cite{Brezis_2002}, \cite{Brezis_VanSchaftingen_Yung_2021Lorentz} on how to recognize constant functions,  a natural question arises whether  the hypothesis on the local integrability of $\nabla u$ in the corollary could be relaxed; one can ask whether the constancy conclusion  holds for all locally integrable functions satisfying $\nu_0(E_{\la,0}[u])<\infty$ for all $\la>0$. However the following example shows that such an extension  fails (for details see \cref{lem:domainsprecise}). 

\begin{examplebf} 
\label{lem:domains} 
Let $\Omega\subset \bbR^N$ be a bounded Lipschitz domain and let $u$ be the characteristic function of $\Omega$. Then $u\in{BV}(\bbR^N) \setminus \dot W^{1,1}(\bbR^N)$  and  $\sup_{\la>0} \la\,  \nu_0(E_{\la,0}[u])<\infty$.
\end{examplebf}

\subsubsection*{More on counterexamples} We now make more explicit the   exclusion of the parameters $\gamma\in [-1,0)$ in  part (\ref{itemc}) of  \cref{theorem_weighted_reverse}  and in \eqref{eq:Marc-1}.
  We shall show in \cref{sec:Cantor}   that  for $\gamma \in (-1,0)$ these negative results can  be related to  self-similar Cantor subsets of $\bbR$, of  dimension $1+\gamma$. 
\begin{theorem} \label{thm:counterexamples} 
Suppose $N\ge 1$. Then the following hold.
\begin{enumerate}[(i)]
\item  Let  $-1\le \gamma<0$. There exists a $C^\infty$ function $u\in \dot W^{1,1}(\bbR^N)$, rapidly decreasing  as $|x|\to \infty$ and such that
\begin{equation} \label{eq:lim_weighted_neg-infty}
\lim_{\lambda \searrow 0} \lambda \nu_\gamma(E_{\la, \gamma}[u]) = \infty.
\end{equation}
\item Let $-1 \leq \ga < 0$. There exists a compactly supported $u\in W^{1,1}(\R^N)$ for which \eqref{eq:lim_weighted_neg-infty} holds. The set  
\[
\big\{u\in W^{1,1} (\R^N) : \limsup_{\la \searrow 0} \la \nu_\gamma(E_{\la, \gamma}[u]) <\infty \big\} 
\]
is meager in $W^{1,1}(\R^N)$, i.e.\ of first category in the sense of Baire. 
\item Let $-1\le  \gamma<0$, $N\ge 2$ or $-1<\gamma<0$, $N=1$. There exists a compactly supported $u\in W^{1,1}(\bbR^N)$ such  that $\nu_\gamma(E_{\la, \gamma}[u]) = \infty$  for all $\la >0$; moreover, the  set 
\[
\big\{u\in W^{1,1}(\R^N): \nu_\gamma(E_{\la, \gamma}[u]) <\infty \text{ for some   $\la\in (0,\infty)$} \big\}
\]  
is meager in $W^{1,1}(\R^N)$. 
 \end{enumerate}
 \end{theorem}
 
 The case $N=1=-\gamma$ plays a special role and is  excluded in the strongest statement  (iii)  since  for  all compactly supported $u\in \dot W^{1,1}(\bbR)$ one has $\nu_{-1}(E_{\la, -1}[u])<\infty$ for all $\la>0$ (\cf. \cref{lem:uniform-cont} below). The proofs of existence of  counterexamples are constructive and the Baire category statements will be obtained as rather straightforward  consequences of the constructions. 
 
\subsection*{Outline of the paper.} 
In \cref{sec:Sobolevbounds} we provide the upper bounds for $[Q_{\gamma/p} u]_{L^{p,\infty}(\R^{2N},\nu_\ga)}$, i.e.\ the proof of inequalities \eqref{eq:Marc-p} and \eqref{eq:Marc-1} in \cref{theorem_weighted_pgr1,theorem_weighted_pequal1}. We first derive these for a dense subclass, relying on covering lemmas, and then extend in \cref{sec:extending-to-Sobolev,sec:BV-extension} to general $\dot{W}^{1,p}$ and $\dot {BV}$-functions.
In \cref{sec:proofoflimits} we derive the limit formulas of \cref{theorem_weighted_reverse}; specifically in \cref{sec:proofofliminfs} we prove the sharp lower bounds involving a $\liminf \la^p\nu_\ga(E_{\la,\ga/p}[u])$ for general functions in $\dot W^{1,p}$ and in \cref{sec:proofoflimsups} we obtain the sharp upper bounds for $\limsup \la^p\nu_\ga(E_{\la,\ga/p}[u])$, under the assumption that $u\in C^1$ is compactly supported. Then in \cref{sec:proofoflimitsW11} we extend these limits to general $\dot W^{1,p}$ functions. In \cref{sec:BV-limit}   we show that  the limit formulas  for $\dot W^{1,1}$ do not extend to general $\dot{BV}$ functions and prove \cref{prop:failureoflimitforBV}. In \cref{sec:backtoSobolev} we prove the reverse inequalities  \eqref{eq:converse-p} and \eqref{eq:converse-1} in \cref{theorem_weighted_pgr1,theorem_weighted_pequal1}. In \cref{sec:gamma=0} we prove \cref{thm:counterexamples-gamma=0} on a characterization of the Lipschitz norm and also discuss \cref{lem:domains}. In \cref{sec:lower-bounds} we provide various constructions of counterexamples and in particular prove  \cref{thm:counterexamples}. We discuss some further perspectives and open problems in \cref{sec:perspectives}.

\subsubsection*{\it Acknowledgements} A.S. and P.-L.Y.  would like to thank the  Hausdorff Research Institute of Mathematics and the organizers of the trimester program ``Harmonic Analysis and Analytic Number Theory'' for a pleasant working environment in the summer of 2021. The research was supported in part by NSF grants DMS-1764295, DMS-2054220 (A.S.) and  by a Future Fellowship FT200100399 from the Australian Research Council (P.-L.Y.).

\section{Bounding $[Q_{\ga/p} u]_{L^{p,\infty}(\R^{2N},\nu_{\ga})}$ by the Sobolev norm.} \label{sec:Sobolevbounds} 

In this section we prove  inequalities \eqref{eq:Marc-p} and \eqref{eq:Marc-1} in Theorems \ref{theorem_weighted_pgr1} and \ref{theorem_weighted_pequal1}.

\subsection{The bound \eqref{eq:Marc-p}  via the  Hardy-Littlewood maximal operator}

Following   \cite{Brezis_VanSchaftingen_Yung_2021}, one can prove the result of \cref{theorem_weighted_pgr1} for $p>1$ by an elementary argument involving the Hardy-Littlewood maximal function  $M\lvert \nabla u\rvert $ of $\lvert \nabla u\rvert$; however the behavior of the constants as $p\searrow 1$  will only be sharp in the range $-1\le \gamma<0$. 

\begin{proposition} \label{proposition_Mp_gamma_-1_0_p}
Let  $N \geq 1$ and  $1 < p < \infty$. There exists a constant $C_N$ such that for all $\gamma\neq 0$ and all $u \in \dot W^{1,p}(\R^N)$, 
\begin{equation} \label{eq:Mp_maximal_bdd}
\sup_{\lambda > 0} \lambda^p \nu_\ga(E_{\la,\ga/p}[u]) \leq \frac{C_N}{|\gamma|} \Big(\frac{ p}{p-1}\Big)^p \norm{\nabla u}_{L^p(\R^N)}^p.
\end{equation}
\end{proposition}

\begin{proof} 
We assume first that  $u\in C^1$ and that $\nabla u$ is compactly supported.
As in \cite[Remark 2.3]{Brezis_VanSchaftingen_Yung_2021},  one 
uses the Lusin-Lipschitz inequality
\Be\label{eq:LL}\frac{|u(x)-u(y)|}{|x-y|} \le C [M(\abs{\nabla u})(x)+M(\abs{\nabla u})(y) ]\Ee
and observes that \eqref{eq:LL} implies
\begin{align*}
E_{\lambda,\gamma/p}[u] \subseteq \{|x-y|^{\gamma/p} < 2C \lambda^{-1} M(|\nabla u|)(x)\} \cup \{|x-y|^{\gamma/p} < 2C \lambda^{-1} M(|\nabla u|)(y)\}.
\end{align*} 
As a consequence 
\[\nu_\ga(E_{\la,\ga/p}[u]) \le 2
\int_x \int_{|h|^{\gamma} < 2C [\la^{-1} M(|\nabla u|)(x)]^p} |h|^{\ga-N} \dif h \dif x.
\]
Direct computation of the inner integral  (distinguishing the cases $\gamma>0$ and $\gamma<0$) yields
\[ 
  \nu_\ga(E_{\la,\ga/p}[u]) \lc_N C^p |\ga|^{-1}\la^{-p} \int_{\Rset^N} [M(|\nabla u|)(x)]^p \dif x.
\]
Inequality  \eqref{eq:Mp_maximal_bdd} follows 
then   from the standard maximal inequality 
$\|Mf\|_p^p\le [C(N)p']^p\|f\|_p^p$ for $p>1$, see \cite{stein-diff} (here $p' = p/(p-1)$). The extension to general $\dot W^{1,p}$ functions will be taken up in 
\cref{sec:extending-to-Sobolev}.
\end{proof}

\subsection{The case $\gamma\in \bbR\setminus [-1,0]$} 

We shall prove the following more precise versions of the 
estimates \eqref{eq:Marc-p} and \eqref{eq:Marc-1} when  $\gamma\notin [-1,0]$, with constants that stay bounded as $p\searrow 1$, indeed we  cover all  $p\in [1,\infty)$.
We denote by $\sigma_{N-1}$  the surface area of the sphere $\bbS^{N-1}$. In the proof of the following theorem we will first establish the estimates for functions $u\in C^1(\bbR^N)$ whose gradient is compactly supported. 
The extension to  $\dot W^{1,p}$ and $\dot{BV}$  will be taken up in \cref{sec:extending-to-Sobolev} and \cref{sec:BV-extension}.

\begin{theorem} \label{theorem_weighted_upper_Lp_allgamma}
There exists an absolute constant $C > 0$ such that for every $N \geq 1$, every $1 \leq p < \infty$, and every  $u\in  \dot W^{1,p} (\bbR^N)$ 
\begin{enumerate}[(a)]
\item if $\gamma > 0$, then
\begin{equation}
\sup_{\lambda > 0} \lambda^p  \nu_\ga(E_{\lambda,\gamma/p}[u]) 
 \leq C \sigma_{N-1} \frac{5^{\gamma}}{\gamma} \norm{\nabla u}_{L^p(\R^N)}^p;
\end{equation}
\item if $\gamma < -1$, then
\begin{equation}
\sup_{\lambda > 0} \lambda^p  \nu_\ga(E_{\lambda,\gamma/p}[u]) 
\leq \frac{C \sigma_{N-1}}{|\gamma|} \Big( 1 + \frac{1}{|\gamma + 1|} \Big) \norm{\nabla u}_{L^p(\R^N)}^p.
\end{equation}
\end{enumerate}
When $p = 1$ the above assertions hold for all
$u \in \dot{BV}(\R^N)$ provided that  $\|\nabla u\|_{L^1(\R^N)}$ is replaced by $\|\nabla u\|_{\mathcal{M}}$. 
\end{theorem}

The proof of \cref{theorem_weighted_upper_Lp_allgamma} relies on the following proposition, in which  $[x,y] \subset \Rset^N$ denotes the closed line segment connecting two points $x,y\in \bbR^N$.

\begin{proposition} \label{proposition_weighted_upper_L1}
Let 
\begin{equation}
E(f,\gamma) \defeq \Big\{(x,y) \in \R^N \times \R^N \colon x \ne y, \, \int_{[x,y]} |f| \dif s > |x-y|^{\gamma+1} \Big\}
\end{equation}
for $f \in C_c(\R^N)$.
There exists an absolute constant $C > 0$ such that for every $N \geq 1$ and every $f \in C_c(\R^N)$,
\begin{enumerate}[(i)]
\item if $\gamma > 0$, then 
\begin{equation} \label{eq:upper_bdd_1}
\iint_{E(f,\gamma)} |x-y|^{\gamma-N} \dif x  \dif y \leq C \sigma_{N-1} \frac{5^{\gamma}}{\gamma} \|f\|_{L^1(\R^N)}; 
\end{equation}
\item if $\gamma < -1$, then 
\begin{equation} \label{eq:upper_bdd_2}
\iint_{E(f,\gamma)} |x-y|^{\gamma-N} \dif x  \dif y \leq \frac{C \sigma_{N-1}}{|\gamma|}  \Big( 1 + \frac{1}{|\gamma + 1|} \Big) \|f\|_{L^1(\R^N)}.
\end{equation}
\end{enumerate}

\end{proposition}

Indeed, to deduce  \cref{theorem_weighted_upper_Lp_allgamma} from \cref{proposition_weighted_upper_L1} one argues as  in the proof of \eqref{eq:new_char_sob_norm} in \cite{Brezis_VanSchaftingen_Yung_2021};  for $u \in C^1(\R^N)$ and $1 \leq p < \infty$, one has
\[
|u(x)-u(y)|^p \leq \Big[ \int_{[x,y]} |\nabla u|\dif s \Big]^p \leq  \int_{[x,y]} |\nabla u|^p \dif s\, |x-y|^{p-1}
\]
for all $x, y \in \R^N$, which implies that
\[
E_{\lambda,\gamma/p}[u] \subseteq E(\lambda^{-p} |\nabla u|^p,\gamma).
\]
Hence for $u \in C^1(\R^N)$ whose gradient is compactly supported, one establishes \cref{theorem_weighted_upper_Lp_allgamma} by applying \cref{proposition_weighted_upper_L1} with $f \defeq \lambda^{-p} |\nabla u|^p$. The extension to $u \in \dot W^{1,p}$ will be taken up in \cref{sec:extending-to-Sobolev}.

\begin{proof}[Proof of \cref{proposition_weighted_upper_L1}]
As in the proof of \cite[Proposition 2.2]{Brezis_VanSchaftingen_Yung_2021}, using the method of rotation, we only need to prove \cref{proposition_weighted_upper_L1} for $N = 1$. Indeed, 
\[
\iint_{E(f,\gamma)} |x-y|^{\gamma-N} \dif x \dif y = \frac{1}{2} \int_{\Sset^{N-1}} \int_{\omega^{\perp}} \iint_{E(f_{\omega,x'},\gamma)} |r-s|^{\gamma-1} \dif r \dif s \dif x' \dif \omega
\]
where for every $\omega \in \Sset^{N-1}$ and every $x' \in \omega^{\perp}$, $f_{\omega,x'}$ is a function of one real variable defined by
\[
f_{\omega,x'}(t) \defeq f(x'+t\omega).
\]
The innermost double integral can be estimated by the case $N=1$ of \cref{proposition_weighted_upper_L1}, and 
\[
\int_{\Sset^{N-1}} \int_{\omega^{\perp}} \int_{\R} |f_{\omega,x'}(t)| \dif t \dif x' \dif \omega = \sigma_{N-1} \|f\|_{L^1(\R^N)}.
\] 
Thus from now on, we assume $N = 1$ and $f \in C_c(\R)$.

If $\gamma > 0$, the desired estimate \eqref{eq:upper_bdd_1} is the content of \cite[Proposition 2.1]{Brezis_VanSchaftingen_Yung_2021}.
On the other hand, suppose now $\gamma < -1$. Without loss of generality, assume $f \geq 0$ on $\R$. 
In addition, we may assume that $f$ is not identically zero, for otherwise there is nothing to prove. 

Let
\[
E_+(f,\gamma) \defeq \{(x,y) \in E(f,\gamma) \colon y < x\}.
\]
Then by symmetry,
\[
\iint_{E(f,\gamma)} |x-y|^{\gamma-1} \dif x \dif y = 2 \iint_{E_+(f,\gamma)} |x-y|^{\gamma-1} \dif x \dif y,
\]
and it suffices to estimate the latter integral.

In what follows we will need to always keep in mind that in view of our assumption  $\gamma<-1$ we have  \(-(\gamma+1)=|\gamma|-1>0.\) We will now  use a simple stopping time argument based on the fact that for all $c\in \bbR$  the continuous function
\[ x \mapsto (x-c)^{-(\gamma+1)}\int_c^x f(s) \dif s,  \quad x\ge c\]
increases from $0$ to $\infty$ on $[c,\infty).$  

Assume that \(\supp \,f \subseteq [a, b]\).
We construct a finite sequence of intervals $I_1, \dots, I_K$, that are disjoint up to end-points, that  cover $\text{supp}\, f = [a,b]$, and that satisfy
\begin{equation} \label{eq:I_idef1}
|I_i|^{-(\gamma+1)} \int_{I_i} f  = \frac{1}{2} \quad \text{for $1 \leq i \leq K$}.
\end{equation}
Indeed, we may take $a_1 \defeq a$, and $a_2 > a_1$ to be the unique number for which \[(a_2-a_1)^{-(\gamma+1)} \int_{a_1}^{a_2} f = 1/2,\] and set $I_1 \defeq [a_1,a_2]$. If $a_2 < b$, we may now repeat, and take $I_2 \defeq [a_2,a_3]$ where $a_3 > a_2$ is the unique number for which $(a_3-a_2)^{-(\gamma+1)} \int_{a_2}^{a_3} f  = 1/2$. Note that the $a_i$'s chosen as such satisfy 
\[
(a_{i+1}-a_i)^{-(\gamma+1)} \geq \frac{1}{2} \|f\|_{L^1(\R)}^{-1},
\]
so that $a_{i+1}-a_i \geq (2 \|f\|_{L^1(\R)})^{1/(\gamma + 1)}$.
This shows that in finitely many steps, we would reach $a_{K+1} \geq b$ for some $K \geq 1$,  with  $a_K<b$ if $1\le K$. Then we have our sequence of disjoint (up to endpoints) intervals $I_1, \dots, I_K$ that cover $[a,b]$ and satisfy \eqref{eq:I_idef1}. We also write $I_0 \defeq (-\infty,a_1]$
and \(I_{K + 1} \defeq [a_{K + 1}, +\infty)\).

We now claim that \(I_i \times I_i \cap E_+(f,\gamma) = \emptyset\) for every $0 \leq i \leq K + 1$.
This being trivially the case when \(i \in \{0, K + 1\}\), we consider the case \(i \in \{1, \dotsc, K\}\): any \(x, y \in I_i\) satisfy
\[
|x-y|^{-(\gamma+1)} \abs[\bigg]{\int_y^x f} \leq |I_i|^{-(\gamma+1)} 
\int_{I_i} f = \frac{1}{2} < 1.
\]
It follows thus that 
\begin{equation}
\label{eq_eik1ahbu6yaog1OobaiK1ied}
 E_+(f,\gamma) =
 \bigcup_{i = 1}^{K + 1} E_+ (f, \gamma) \cap ((a_i, +\infty) \times (-\infty, a_i))
\end{equation}

Furthermore, for \(i \in \{2, \dotsc, K\}\), if $y < a_i < x$ and $x-y < \min\{|I_i|,|I_{i-1}|\}$, then
\begin{align*}
|x-y|^{-(\gamma+1)} \abs[\bigg]{\int_y^x f} < \min\{|I_i|,|I_{i-1}|\}^{-(\gamma+1)} \Big( \int_{I_{i-1}} f+\int_{I_i} f\Big) &\\
 \leq |I_{i-1}|^{-(\gamma+1)} \int_{I_{i-1}} f + |I_i|^{-(\gamma+1)} \int_{I_i} f 
 \leq \frac{1}{2} + \frac{1}{2} = 1,&
\end{align*}
(again we used $\gamma < -1$ so that $-(\gamma + 1) > 0$ here), from which it follows that $(x,y) \not \in E_+(f,\gamma)$. 
Combining this with a similar argument for \(i \in \{1, K + 1\}\), we get that if \((x, y) \in E_+ (f,\gamma) \cap (a_i, +\infty) \times (-\infty, a_i)\), then \(\abs{x - y} \ge \min \{\abs{I_i}, \abs{I_{i - 1}}\}\),
and thus 
\begin{equation*}
\begin{split}
& \int_{E_+ (f, \gamma) \cap (a_i, +\infty) \times (-\infty, a_i)}
 \abs{x - y}^{\gamma - 1} \dif x \dif y
\\&\quad \le
 \int_{a_i}^\infty \int_{-\infty}^{\min\{a_i, x - \min \{\abs{I_i}, \abs{I_{i - 1}}\}\}} \abs{x - y}^{\gamma - 1} \dif y \dif x
 \\&\quad = \frac{1}{\abs{\gamma}} \int_{a_i}^\infty (\max\{x - a_i, \min \{\abs{I_i}, \abs{I_{i - 1}}\}\})^\gamma \dif x\\
 &\quad = \frac{1}{\abs{\gamma}}\brk*{1 + \frac{1}{\abs{\gamma + 1}}} \min\{\abs{I_i}, \abs{I_{i - 1}}\}^{\gamma + 1}
 \le \frac{2}{\abs{\gamma}}\brk*{1 + \frac{1}{\abs{\gamma + 1}}}
 \int_{I_{i - 1} \cup I_{i}} f.
 \end{split}
\end{equation*}
(The computation of these integrals uses our assumption  $\gamma + 1 < 0$.)
Summing the estimates, we get in view of \eqref{eq_eik1ahbu6yaog1OobaiK1ied}
\begin{equation*}
\begin{split}
 \int_{E_+ (f, \gamma)}
 \abs{x - y}^{\gamma - 1} \dif x \dif y
\le \frac{4}{\abs{\gamma}}\brk*{1 + \frac{1}{\abs{\gamma + 1}}}
  \int_{\Rset} f.
 \end{split}
\end{equation*}
 We have thus completed the proof of \eqref{eq:upper_bdd_2} under the assumption $\gamma < -1$ and \(N = 1\).
\end{proof}

\subsection{Proof of \cref{proposition_Mp_gamma_-1_0_p} and \cref{theorem_weighted_upper_Lp_allgamma} for  general $\dot W^{1,p}$ functions}
\label{sec:extending-to-Sobolev}

We use a limiting argument, together with the following fact: if $u \in \dot{W}^{1,p}(\R^N)$, $N \ge 1$, and $1 \leq p < \infty$, then there exists a Lebesgue measurable set $X \subset \R^{2N}$ with $\mathcal{L}^{2N}(X) = 0$, so that for every $(x,h) \in \R^{2N} \setminus X$, we have
\begin{equation} \label{eq:Sob_FTC}
u(x+h)-u(x) = \int_0^1 \langle h, \nabla u(x+th) \rangle \dif t.
\end{equation}
Indeed, both sides are measurable functions of $(x,h) \in \R^{2N}$, and if $X$ is the set of all $(x,h)$ where the two sides are not equal, then $X$ is a measurable subset of $\R^{2N}$, and the assertion will follow from Fubini's theorem if for every fixed $h \in \R^N$, we have $\cL^N (\{x \in \R^N \colon (x,h) \in X\}) = 0$, i.e.\ \eqref{eq:Sob_FTC} holds for $\cL^N$ almost every $x$. This follows since for every $\phi \in C^{\infty}_c(\R^N)$, one has
\begin{align*}
&\int_{\R^N} [u(x+h)-u(x)] \phi(x) \dif x = \int_{\R^N} u(x) [\phi(x-h) - \phi(x)] \dif x \\
&= -\int_{\R^N} u(x) \int_0^1 \langle h, \nabla \phi(x-th) \rangle \dif t \dif x = \int_{\R^N} \int_0^1 \langle h, \nabla u(x) \rangle \phi(x-th) \dif t \dif x \\
&= \int_{\R^N} \int_0^1 \langle h, \nabla u(x+th) \rangle \dif t \, \phi(x) \dif x.
\end{align*}

Now given \(u \in \dot{W}^{1, p} (\Rset^N)\), there exists a sequence \(u_n \in C^\infty (\Rset^N)\) such that $\nabla u_n$ are compactly supported, and 
\begin{equation}\label{eq:convergence}
 \|\nabla(u_n-u) \|_{L^p(\bbR^N)}  \to 0.
\end{equation}
Indeed if $N > 1$ and $p \geq 1$, or if $N = 1$ and $p > 1$, then this follows from the density of $C^{\infty}_c(\R^N)$ in $\dot{W}^{1,p}(\R^N)$ as asserted in \cite{hajlasz-kalamajska} (in this case one may choose $u_n \in C^{\infty}_c(\R^N)$). The density of $C^{\infty}_c(\R^N)$ in $\dot{W}^{1,p}$ {\it fails} when $N = p = 1$ (again see \cite{hajlasz-kalamajska}); the issue is that if $\nabla u$ is supported in a convex set in $\bbR^N$, $N\ge 2$, then $u$ is constant in the complement of the set, but this fails for $N=1$ since the complement of a bounded interval has two connected components. On the other hand, in the anomalous case $N = 1$ and $p = 1$, one can choose an approximation of the identity to get a sequence $v_n$ of $C^\infty_c$ functions on $\R$ such that $\|v_n-u'\|_{L^1(\R)} \to 0$. One can then take $u_n(x)\defeq \int_{0}^x v_n(t) \dif t$, and \eqref{eq:convergence} follows with $u_n'=v_n$ being compactly supported (even though $u_n$ may not be compactly supported).

 Let, for $R>1$,
 \[
 K_R=\{(x,y)\in \bbR^{2N}:  \, |x|\le R, \,|y|\le R
 \text{ and } R^{-1}\le |x-y| \}.
 \]
By monotone convergence it suffices to prove 
\Be\label{eq:bound-on-KR}
  \nu_\ga(E_{\lambda,\gamma/p} [u] \cap K_R)
    \le C \frac{\norm{\nabla u}_{L^p (\Rset^N)}^p}{\lambda^p}.
\Ee
with $C$ independent of  $R$. 

Under the assumptions of \cref{proposition_Mp_gamma_-1_0_p} and \cref{theorem_weighted_upper_Lp_allgamma} on $p$ and $\ga$, since $u_n \in C^{\infty}_c(\R^N)$, we already know 
\[
  \nu_\ga(E_{\lambda,\gamma/p} [u_n])
    \le C \frac{\norm{\nabla u_n}_{L^p (\Rset^N)}^p}{\lambda^p}.
\]
Moreover, the sequence $Q_{\gamma/p} u_n$ converges to  $Q_{\ga/p}u$ in $L^p(K_R)$ as $n \to \infty$.
Indeed, using \eqref{eq:Sob_FTC} we may write 
\[
Q_{\gamma/p} u(x,y) = \frac{1}{|x-y|^{\ga/p}} \int_0^1 \biggl\langle \frac{x-y}{|x-y|}, \nabla u((1-t)y+tx) \biggr\rangle \dif t
\]
for $\mathcal{L}^{2N}$ a.e.\ $(x,y) \in \R^{2N}$, and similarly for $u_n$ in place of $u$, which allows us to estimate
\begin{align*}
  & \Big(\iint_{K_R}  |Q_{\ga/p}u_n(x,y)-Q_{\ga/p}u (x,y)|^p \dif x\dif y \Big)^{1/p} 
  \\&\le 
  R^{\frac\ga p} \int_0^1\Big(\int_{|x|\le R}\int_{|y|\le R} |\nabla (u_n-u)((1-s) x+ sy ) |^p \dif x\dif y\Big)^{1/p} \dif s
  \\
  &\le 2^{N/p} (2R)^{N/p} R^{\ga/p} 
  \|\nabla (u_n-u_{n+1}) \|_p \to 0.
\end{align*}
By passing to a subsequence if necessary, we may assume that $Q_{\gamma/p} u_n$ converges  $\cL^{2N}$-a.e.\ to  $Q_{\ga/p}u$ on $K_R$ as $n \to \infty$. Thus 
\[
K_R\cap E_{\lambda,\gamma/p} [u]
\subseteq  K_R\cap \Big(\bigcup_{n \in \Nset} \bigcap_{\ell \ge n}
E_{\lambda, \gamma/p}[u_\ell]\Big) 
\]
which implies
\[
\begin{split} \nu_\ga(K_R\cap E_{\lambda,\gamma/p} [u])
&\le \lim_{n \to \infty} \nu_{\gamma}\biggl(K_R\cap \bigcap_{\ell \ge n}
E_{\lambda, \gamma/p}[u_\ell]\biggr) \le 
\liminf_{n \to \infty}\nu_\ga (K_R\cap E_{\lambda,\gamma/p} [u_n]) \\
  &\le C \liminf_{n \to \infty} \frac{\norm{\nabla u_n}_{L^p (\Rset^N)}^p}{\lambda^p}
  \le C \frac{\norm{\nabla u}_{L^p (\Rset^N)}^p}{\lambda^p}.
\end{split}
\]

\subsection{\it Proof of  \cref{theorem_weighted_upper_Lp_allgamma} for $\dot{BV}$-functions} \label{sec:BV-extension}

We choose a sequence  $\rho_n\in C^\infty_c(\Rset^N)$ with $\rho_n=2^{nN}\rho(2^n\cdot)$ and $\int_{\Rset^N} \rho\,\dif x=1$ and set $u_n\coloneqq u*\rho_n$. Then $u_n\in \dot W^{1,1}(\Rset^N)$ and $u_n\to u$ almost everywhere. This means if $G_L \coloneqq \{(x,h) \in \Rset^N \times \Rset^N : |x|\le L, \,L^{-1}\le |h|\le L\}$ then
\[ 
\lim_{L\to\infty} \nu_\gamma (E_{\la,\ga}[u_n]\cap G_L)  =
\nu_\ga(E_{\la,\ga}[u]\cap G_L),
\]
by dominated convergence. Also
\[
\begin{split}
&\|\nabla u_n\|_{L^1(\Rset^N)} =\sup_{\substack{\vec \phi\in C^\infty_c\\ \|\phi\|_\infty\le 1}} \Big|\int u_n (x)\,\mathrm{div} \vec\phi(x)\dif x\Big|=
\\&=\sup_{\substack{\vec \phi\in C^\infty_c\\ \|\phi\|_\infty\le 1}} \Big|\int u (x) \,\mathrm{div} (\rho_n*\vec\phi)(x)\dif x\Big|
\le \|\nabla u\|_{\mathcal{M}};
\end{split}
\]
here we used  $\|\rho_n*\vec\phi\|_\infty\le \|\vec\phi\|_\infty$ for the last inequality. Combining these two limiting identities with \cref{theorem_weighted_upper_Lp_allgamma} we get the desired inequalities with $E_{\la,\ga}[u] $ replaced by $E_{\la,\ga}[u]\cap G_L$. By monotone convergence we may finish the proof letting $L\to\infty$.
\qedhere

\section{Proof of \cref{theorem_weighted_reverse}}
\label{sec:proofoflimits}
We extend and refine arguments from \cite{BN2018},  \cite{Brezis_VanSchaftingen_Yung_2021} which are partially inspired by techniques developed in \cite{Bourgain_Brezis_Mironescu_2001}.
\subsection{\it A Lebesgue differentiation lemma}
 Our argument uses the following standard variant of the Lebesgue differentiation theorem. For lack of a proper reference, a proof is provided for the convenience of the reader.
 \begin{lemma}\label{lem:Lebesguevariant} 
 Let $u\in \dot W^{1,1}(\bbR^N)$ and let $\{\delta_n\}$ be a sequence of positive numbers with $\lim_{n\to\infty} \delta_n=0$. Then 
 \[
 \lim_{n\to \infty}  \frac{u(x+\delta_n h)-u(x)}{\delta_n} = \inn{h}{\nabla u(x)}
 \]
 for almost every $(x,h)\in \bbR^N\times \bbR^N$.
 \end{lemma} 
 \begin{proof} 
 If $u\in C^1$ with compact support the limit relation  clearly holds for all $(x,h)$. We shall below consider for each $\theta\in \bbS^{N-1}$ consider the maximal function
\[ 
\fM_\theta F(x)= \sup_{t>0} \frac{1}{t}  \int_0^t |F(x+r\theta) | \dif r
\] 
which is well defined for all $\theta$, a measurable function on $\bbR^N \times \bbS^{N-1}$, and satisfies a weak type $(1,1)$ inequality
\[
\cL^N(\{x\in \bbR^N: \fM_\theta F(x)>a\})\le 5a^{-1} \|F\|_1.
\]
 
 Let $u \in \dot{W}^{1,1}(\R^N)$ and $\cA_M=\{h\in \bbR^N: 2^{-M}\le |h|\le 2^M\}$. It suffices to prove the limit relation for almost  every $(x,h)\in \bbR^N\times \cA_M$.
 From \eqref{eq:Sob_FTC} we get that for every $n \ge 1$,
 \[
 \frac{u(x+\delta_n h)-u(x)}{\delta_n} = \frac 1{\delta_n|h|}\int_0^{\delta_n|h|} \inn{h}{\nabla u(x+ r\tfrac{h}{|h|}} ) \dif r
 \]
 for $\mathcal{L}^{2N}$ almost every $(x,h) \in \R^N \times \cA_M$; as a result, there exist representatives of $u$, $\nabla u$ and  a null set $\cN\in \bbR^N\times\cA_M$ such that the identity holds for all $(x,h)\in \cN^\complement$ and all $n \geq 1$. 
 It suffices to show that for every $\alpha>0$, $\epsilon>0$
 \Be\label{eq:wttoshow}
\mathcal{L}^{2N}\Big(\Big\{ (x,h)\in \bbR^N\times \cA_M:  \limsup_{n\to\infty}
 \Big|
 \frac{1}{\delta_n|h|}\int_0^{\delta_n|h|} \inn{h}{\nabla u(x+ rh)} \dif r  - \inn{h}{\nabla u(x)} \Big|>\alpha\Big\}\Big)
 \le \epsilon.
 \Ee

Let $v\in C^1_c$ so that $\|\nabla(v-u)\|_1\le \alpha \epsilon / (12 \mathcal{L}^N (\cA_M))$.
Let $g=u-v$. Since the asserted limiting relation holds for $v$, we see that the   expression on the left  hand side of \eqref{eq:wttoshow} is dominated by
\begin{align*} 
&\mathcal{L}^{2N}\Big(\big\{ (x,h)\in \bbR^N\times \cA_M: |\nabla g(x)|+\sup_{n>0}  \frac{1}{\delta_n|h|} \int_0^{\delta_n|h|} |\nabla g(x+ r\tfrac{h}{|h|}) |\dif r  >\alpha\big\}\Big)
 \\ &\le 2 \mathcal{L}^N(\cA_M) \alpha^{-1}  \|\nabla g\|_1
 + \int_{\cA_M}  \cL^N(\{x: \fM_{h/|h|} |\nabla g|(x) > \alpha/2\}) \dif h
\\&  \le 12 \mathcal{L}^N(\cA_M) \alpha^{-1} \|\nabla g\|_1
 \le \epsilon
 \end{align*}
 since $\|\nabla g\|_1\le \alpha \epsilon / (12 \mathcal{L}^N (\cA_M))$.
 \end{proof}

\subsection{\it The lower bounds for $\liminf \la^p\nu_\ga(E_{\la,\ga/p} [u])$}
\label{sec:proofofliminfs}

We use \cref{lem:Lebesguevariant} to establish lower bounds, relying on an idea in  \cite{BN2018} where the case $\ga=-1$ was considered.
\begin{lemma} \label{lem:liminfs}
Let $1\le p<\infty$ and $u\in \dot W^{1,p} (\bbR^N)$.
Then 
\begin{enumerate}[(i)]
\item For $\gamma>0$ 
\[ 
\liminf_{\la\to \infty} \la^p \nu_\gamma(E_{\la,\ga/p}[u]) \ge \frac{\ka(p,N)}{|\ga|} \|\nabla u\|_{L^p(\bbR^N)}^p
\]
\item For $\gamma<0$ 
\[ 
\liminf_{\la\searrow 0} \la^p \nu_\gamma(E_{\la,\ga/p}[u]) \ge \frac{\ka(p,N)}{|\ga|} \|\nabla u\|_{L^p(\bbR^N)}^p
\]
\end{enumerate}
\end{lemma}

\begin{proof}
We write, for $\la>0$ and $\delta>0$
  \begin{align*} \la^p \nu_\ga(E_{\la,\ga/p}[u])&=\la^p \iint_{\frac{|u(x+h)-u(x)|}{|h|^{1+\ga/p}}> \lambda } |h|^{\ga-N}  \dif h\dif x
 \notag
 \\
 &=\la^p\delta^\gamma
 \iint_{\left|\frac{u(x+\delta h)-u(x)}{\delta |h|}\right|^p> \lambda^p \delta^{\ga} |h|^{\ga} } |h|^{\ga-N}  \dif h\dif x, 
 \end{align*}
 here we have changed variables replacing $h$ by $\delta h$. 
 Hence 
 \Be \label{eq:superlevelset}  
 \la^p \nu_\ga(E_{\la,\ga/p}[u])=
 \iint \bbone_{(|h|^{\ga},\infty)} \big( \big|\tfrac {u(x+\delta h)-u(x)}{\delta|h|} \big|^p   \big) \,  |h|^{\gamma-N} \dif h \dif x  \text{ with } \delta=\la^{-p/\ga}.
 \Ee
 We now take a sequence $\{\la_n\}$ of positive numbers, set $\delta_n =\la_n^{-p/\gamma} $ and note that 
  \Be\label{eq:which-limit} \lim_{n\to\infty} \delta_n=0 \text { if }
  \begin{cases} \lim_{n\to\infty}\la_n=\infty \text { and } \gamma>0,  
  \\
  \lim_{n\to\infty}\la_n =0 \text{ and } \gamma<0.\end{cases} 
  \Ee
  Also observe that 
 \[
 \liminf_{n\to \infty} \bbone_{(|h|^{\ga},\infty)}(s_n) \ge \bbone_{(|h|^{-\ga},\infty)}(t) \text{ if } \lim_{n\to\infty} s_n= t.
 \]
 Now assume that  $\la_n\to \infty$ if $\ga>0$ and $\la_n\to 0^+$ if $\ga<0$ and stay with $\delta_n=\la_n^{-p/\ga}$, a sequence which converges to $0$ in both cases.   Use Fatou's lemma in \eqref{eq:superlevelset}  and combine it with  \cref{lem:Lebesguevariant} to get
 \begin{align*}
 \liminf_{n\to\infty} \la_n^p \nu_\ga(E_{\la_n,\ga/p}[u])
 &\ge 
\iint \liminf_{n\to\infty} \bbone_{(|h|^{\ga},\infty)} \big(\big|\tfrac {u(x+\delta_n h)-u(x)}{\delta_n |h|} \big|^p \big) \,  |h|^{\ga-N } \dif h \dif x  
\\
&\ge 
\iint  \bbone_{(|h|^{\ga},\infty)} \big( \lim_{n\to\infty}\big|\tfrac {u(x+\delta_n h)-u(x)}{ \delta_n |h|} \big|^p \big) \,  |h|^{\ga-N } \dif h \dif x  
\\
&=
\iint_{|h|^{\ga} < \big|\inn{\tfrac h{|h|}}{\nabla u(x)}\big|^p} |h|^{\ga - N} \dif h \dif x   \eqqcolon J_{\ga}.
\end{align*}
We use polar coordinates  $h=r\theta$  and write the last expression as 
\begin{align*}
J_{\gamma}&
= \iint\limits _{\bbR^N\times\bbS^{N-1}}  \int\limits_{
 r^{\gamma} < |\inn{\theta}{\nabla u(x)} |^p } r^{\ga-1}\dif r\dif\theta\dif x
\\&= \frac{1}{|\ga|} \iint_{\bbR^N\times S^{N-1}}
|\inn{\theta}{\nabla u(x)}| ^p  \dif \theta\dif x= \frac{\kappa(p,N) }{|\ga| }
\|\nabla u\|_{L^p(\bbR^N)}^p,
\end{align*} 
with the calculation valid in both cases $\ga>0$ and $\ga<0$. 
\end{proof}

\subsection{\it {Upper bounds for $\limsup \la^p\nu_\ga(E_{\la,\ga/p}[u])$, for $C^1_c$ functions}}
\label{sec:proofoflimsups}

We assume that $u\in C^1$ is compactly supported and obtain the sharp upper bounds for $\limsup_{\la\to \infty} \la^p \nu_\gamma(E_{\la,\ga/p}[u])$ when $\ga>0$ and $\limsup_{\la\to 0} \la^p \nu_\gamma(E_{\la,\ga/p}[u])$ when $\ga<0$.
 
\begin{lemma} \label{lem:limsupsN>1}
Suppose $u\in C^1_c(\bbR^N)$ and $1\le p<\infty$.
Then the following hold. 
\begin{enumerate}[(i)]
\item If  $\gamma>0$ then 
\[
\limsup_{\la\to \infty} \la^p \nu_\gamma(E_{\la,\ga/p}[u]) \le \frac{\ka(p,N)}{|\ga|} \|\nabla u\|_{L^p(\bbR^N)}^p.
\]
\item If $\gamma<0$ then
\[
\limsup_{\la\searrow 0} \la^p \nu_\gamma(E_{\la,\ga/p}[u]) \le \frac{\ka(p,N)}{|\ga|} \|\nabla u\|_{L^p(\bbR^N)}^p.
\]
\item The statement in part (i) continues to hold for $u \in C^1(\R^N)$ whose gradient is compactly supported.
\end{enumerate}
\end{lemma} 

\begin{remark}\label{remark3.4}
The subtlety in part (iii) above is only relevant in dimension $N = 1$, since if $N \ge 2$, then any function in $C^1(\R^N)$ with a compactly supported gradient is constant outside a compact set.
\end{remark}

\begin{proof}[Proof of \cref{lem:limsupsN>1}]
We distinguish the cases $\gamma>0$ and $\gamma<0$.

\noindent{\it The case $\ga>0$.} We assume that $\nabla u$ is compactly supported. To prove part (iii) (and thus part (i)) assume
\Be\label{eq:L-A} \la\ge 
 L \defeq \Big\|\Big(\sum_{i=1}^N|\partial_iu|^2\Big)^{1/2} \Big\|_{L^\infty(\bbR^N)}.
\Ee
Then 
\begin{equation} \label{eq:cond2}
(x,y)\in E_{\la,\ga/p}[u] \implies \la|x-y|^{\ga/p}\le L \implies  |x-y|\le 1.
\end{equation}
Furthermore, if $(x,y) \in E_{\la,\ga/p}[u]$, then writing $y=x+r\om$ with $r > 0$ and $\om\in \bbS^{N-1}$, we have
\begin{equation} \label{eq:Taylor}
\begin{split} 
&\la r^{\ga/p}\le |\nabla u(x)\cdot\om|+\rho(r)
\quad \text{with } \rho(r) \defeq \sup_{x\in \R^N} \sup_{|h|\le r} |\nabla u(x+h)-\nabla u(x) |;
\end{split}
\end{equation} since $\nabla u$ is uniformly continuous on $\R^N$ we have $\rho(r)\searrow  0$  as $r\searrow 0$. This,  together with the first implication of \eqref{eq:cond2}, shows that 
\begin{equation} \label{eq:cond1}
\la r^{\gamma/p} \le |\nabla u(x)\cdot\om|+\rho (( \tfrac L\la)^{p/\ga}).
\end{equation}

Let $B$ be a ball centered at the origin containing $\supp(\nabla u)$, and let $\widetilde B$ the expanded ball with radius $1+\mathrm{rad}(B)$.  Then for $x \notin \widetilde B$, we have $Q_{\ga/p}u(x,y) = 0$ for every $y$ with $|x-y| \leq 1$, and \eqref{eq:cond2} shows $(x,y) \notin E_{\la,\ga/p}[u]$ for every $y$ with $|x-y| > 1$, so $E_{\la,\ga/p}[u] \subseteq \widetilde{B} \times \Rset^N$.
Define, for $x \in \tilde{B}$, $\omega \in \S^{N-1}$, and $\la > 0$ 
\[
\overline R(x,\omega,\la) :=
\left( \la^{-1} ( |\nabla u(x)\cdot\om|+\rho (( \tfrac L\la)^{p/\ga})
\right)^{p/\ga}
\]
Then by \eqref{eq:cond1}, 
\begin{align*} 
\la^p\nu_\ga (E_{\la,\ga/p}[u]) &\le \la^p \int_{\widetilde B}
\int_{\bbS^{N-1}} \int_0^{\overline R(x,\om,\la)}  r^{\ga-1} \dif r\dif \om\dif x
\\
&=\ga^{-1} 
\int_{\widetilde B} \int_{\bbS^{N-1}}
\Big( |\nabla u(x)\cdot\om|+\rho (( \tfrac L\la)^{p/\ga})
\Big)^{p}\dif \om\dif x.
\end{align*} 
Letting $\la\to \infty$ we get 
\[
\limsup_{\la\to\infty} \la^p\nu_\ga (E_{\la,\ga/p}[u]) \le \ga^{-1} \ka(p,N) \int_{\widetilde B}
 |\nabla u(x)|^p\dif x
\]
and hence the assertion.

\noindent{\it The case $\ga<0$.} 
 We first note that if $(x,y) \in E_{\lambda,\gamma/p}[u]$, then writing $y = x+r \omega$, we have again \eqref{eq:Taylor}.
 
 Now let $\eps>0$, and let $\delta(\eps)>0$ be such that $\rho(r)\le \eps$ for $0<r\le \delta(\eps) $.
Let  
\[  
r_{\la}(x,\om,\eps) = \min \Big\{ \delta(\eps), \Big(\frac{ \la}{|\nabla u(x)\cdot \omega|+\eps}\Big)^{\frac{p}{-\ga}}\Big\}.
\]
Note that $r_\la(x,\om,\eps)>0$ for $\la>0$. Also if $(x, x+r\omega)\in E_{\la,\gamma/p}[u]$  then $r\ge r_\la(x,\om,\eps)$; indeed, either $r_\la(x,\om,\eps) \ge \delta(\varepsilon)$ already, or else $r_\la(x,\om,\eps) < \delta(\varepsilon)$ in which case \eqref{eq:Taylor} shows $r_\la(x,\om,\eps) \ge \Big(\frac{ \la}{|\nabla u(x)\cdot \omega|+\eps}\Big)^{\frac{p}{-\ga}}$.

Finally let $B$ be any ball in $\R^N$  containing the support of $u$, and let $\widetilde B$ be the double ball. Then
\begin{align*} 
&\quad \limsup_{\lambda \searrow 0} \lambda^p 
\nu_\ga(E_{\lambda,\gamma/p}[u] \cap (\widetilde B \times \R^N)) 
 \leq \limsup_{\lambda \searrow 0} \lambda^p \int_{\widetilde B} \int_{\Sset^{N-1}} \int_{r_{\lambda}(x, \omega,\eps)}^{\infty} r^{\gamma-1} \dif r \dif \omega \dif x 
 \\  \notag
 &= \limsup_{\lambda \searrow 0} \lambda^p \int_{\widetilde B} \int_{\Sset^{N-1}} \frac{1}{|\gamma|} [r_{\lambda}(x, \omega,\eps)]^\gamma  \dif \omega \dif x 
 \\
 \notag &= \limsup_{\lambda \searrow 0} \frac{1}{|\gamma|}
 \int_{\widetilde B} \int_{\bbS^{N-1}} \max\{ \la^p\delta(\eps)^{\gamma}, (|\nabla u(x)\cdot\omega|+\eps)^p \} \dif \omega \dif x
 \\
 &= \frac{1}{|\gamma|} \int_{\widetilde B} \int_{\Sset^{N-1}} (|\nabla u(x) \cdot \omega|+\eps)^p \dif \omega \dif x.
 \end{align*} 
 Since $\eps>0$ was arbitrary we obtain 
 \Be \label{eq:insideB}
 \limsup_{\lambda \searrow 0} \lambda^p 
\nu_\ga\bigl(E_{\lambda,\gamma/p}[u] \cap (\widetilde B \times \R^N)\bigr) \le 
  \frac{1}{|\gamma|} \ka(p,N) \norm{\nabla u}_{L^p(\R^N)}^p.
\Ee
Since $u = 0$ in $\bbR^N\setminus B$, if $(x,y) \in E_{\lambda,\gamma/p}[u] \cap ((\R^N \setminus \widetilde B) \times \R^N)$ then $y \in B$. Therefore 
\begin{align*}
&\limsup_{\lambda \searrow 0} \lambda^p \nu_\ga\bigl(E_{\lambda,\gamma/p}[u] \cap ((\R^N \setminus \widetilde B) \times \R^N)\bigr)   \leq \limsup_{\lambda \searrow 0} \lambda^p \int_{B} \int_{\R^N \setminus \widetilde B} |x-y|^{\gamma-N} \dif x \dif y
 = 0.
\end{align*}
This finishes the proof of part (ii).
\end{proof}

In dimension $N = 1$, when $\ga < -1$, one can also weaken the hypothesis $u \in C^1_c(\R)$ in  \cref{lem:limsupsN>1} to $u \in C^1(\R)$ and $u'$ is compactly supported:

\begin{lemma} \label{lem:limsupsN=1}
Suppose $u\in C^1(\bbR)$, $u'$ is compactly supported, and $1\le p<\infty$. If $\gamma<-1$ then
\[ 
\limsup_{\la\searrow 0} \la^p \nu_\gamma(E_{\la,\ga/p}[u]) \le \frac{\ka(p,N)}{|\ga|} \|u'\|_{L^p(\bbR)}^p.
\]
\end{lemma} 

\begin{proof}
Let $\supp(u')\subset B\coloneqq(-\beta,\beta)$.
By \eqref{eq:insideB} we have 
\[
\limsup_{\la\searrow 0} \nu_\ga(E_{\lambda,\gamma/p}[u] \cap (-2\beta,2\beta)\times \bbR) 
\leq \frac{1}{|\ga|}\ka(p,1) \|u'\|_{L^p(\R)}^p.
\]
Moreover, since $u$ is constant on $(\beta,\infty)$ and constant on $(-\infty,-\beta)$, if $(x,y) \in E_{\la,\ga/p}[u]$ and $x<-2\beta$ then $y> -\beta$, and if $(x,y) \in E_{\la,\ga/p}[u]$ and $x>2\beta$ then $y< \beta$. Since $\gamma<-1$, 
\begin{align*}
&\nu_\ga(E_{\lambda,\gamma/p}[u] \cap (\bbR\setminus(-2\beta,2\beta))\times \bbR)  
\\&\le \int^\infty_{2\beta} \int_{-\infty}^\beta (x-y)^{\gamma-1} \dif y\dif x 
+\int_{-\infty}^{-2\beta}\int_{-\beta}^\infty (y-x)^{\gamma-1} \dif y \dif x  <\infty. 
\end{align*}
We conclude
\[
\limsup_{\la \searrow 0}  \la^p\nu_\ga(E_{\lambda,\gamma/p}[u] \cap (\bbR\setminus(-2\beta,2\beta))\times \bbR)  =0. \qedhere
\]
\end{proof}

\subsection{\it Upper bounds for $\limsup \la^p\nu_\ga(E_{\la,\ga/p}[u])$, for general $\dot{W}^{1,p}$ functions}\label{sec:proofoflimitsW11}

Let $N \ge 1$, $1 \leq p < \infty$ and  $u \in \dot{W}^{1,p}(\R^N)$. In light of \cref{lem:liminfs}, to prove the limiting relations \eqref{eq:lim_weighted} and \eqref{eq:lim_weighted_neg} in \cref{theorem_weighted_reverse}, we need only show that
\begin{equation} \label{eq:upperbdd1}
\limsup_{\la \to \infty} \la^p\nu_\ga(E_{\la,\ga/p}[u]) \leq \frac{\kappa(p,N)}{|\ga|} \|\nabla u\|_{L^p(\R^N)}^p 
\end{equation}
if $\ga > 0$ and
\begin{equation} \label{eq:upperbdd2}
\limsup_{\la \searrow 0} \la^p\nu_\ga(E_{\la,\ga/p}[u]) \leq \frac{\kappa(p,N)}{|\ga|} \|\nabla u\|_{L^p(\R^N)}^p
\end{equation}
if $\ga < 0$ and $p > 1$, or $\ga < -1$ and $p = 1$.
\cref{lem:limsupsN>1}(i)(ii) asserts that these desired inequalities hold for functions in $C^1_c(\R^N)$. When $N \ge 2$ or $p > 1$, a general $\dot{W}^{1,p}(\R^N)$ function can be approximated in $\dot{W}^{1,p}(\R^N)$ by functions in $C^1_c(\R^N)$: by \cite{hajlasz-kalamajska}, there exists a sequence $\{u_n\}$ in $C^{\infty}_c(\R^N)$ such that $\lim_{n\to\infty} \|\nabla (u_n-u)\|_{L^p(\R^N)}=0.$ If further $\ga > 0$, or $\ga < 0$ and $p > 1$, or $\ga < -1$ and $p = 1$, then by parts (i) of  \cref{theorem_weighted_pgr1,theorem_weighted_pequal1} (proved in \cref{sec:Sobolevbounds}), we have
\begin{equation} \label{eq:upperbddsup}
\sup_{\lambda>0} \lambda^p \nu_\gamma (E_{\lambda, \gamma/p}[u_n-u]) \le C^p_{N, p, \gamma} \norm{\nabla (u_n-u)}_{L^p(\Rset^N)}^p.
\end{equation}
It follows that for every $n$ and every $\delta \in (0,1)$,
\begin{equation} \label{eq:argument}
\begin{split}
\limsup_{\la \to \infty} \lambda^p \nu_\gamma (E_{\lambda, \gamma/p}[u]) 
&\leq \limsup_{\la \to \infty} \lambda^p \nu_\gamma (E_{(1-\delta)\lambda, \gamma/p}[u_n]) + \sup_{\lambda>0} \lambda^p \nu_\gamma (E_{\delta \lambda, \gamma/p}[u_n-u]) \\
&\leq \frac{\kappa(p,N)}{|\ga|(1-\delta)} \|\nabla u_n\|_{L^p(\R^N)}^p + \frac{C_{N,p,\ga}^p \|\nabla (u_n-u)\|_{L^p(\R^N)}^p}{\delta^p}
\end{split}
\end{equation}
if $\ga > 0$, and a similar inequality holds with $\limsup_{\la \to \infty}$ replaced by $\limsup_{\la \searrow 0}$ if $\ga < 0$, $p > 1$ or $\ga < -1$, $p=1$. Letting first $n \to \infty$ and then $\delta \to 0$, we get the desired conclusions \eqref{eq:upperbdd1} and \eqref{eq:upperbdd2} under the corresponding conditions on $\ga$ and $p$.

It remains to tackle the case $N = p = 1$, in which case we only need to prove \eqref{eq:upperbdd1} when $\ga > 0$ and \eqref{eq:upperbdd2} when $\ga < -1$. Using \eqref{eq:convergence}, we approximate $u$ by finding a sequence $\{u_n\}$ in $C^{\infty}(\R)$ so that $u_n'$ are compactly supported for each $n$, and $\lim_{n \to \infty} \|u_n'-u'\|_{L^1(\R)}  = 0$. Since the desired inequalities hold for $u_n$ in place of $u$ by  \cref{lem:limsupsN>1}(iii) and \cref{lem:limsupsN=1}, and since part (i) of \cref{theorem_weighted_pequal1} applies to give \eqref{eq:upperbddsup} when $\ga > 0$ or $\ga < -1$, our earlier argument in \eqref{eq:argument} can be repeated to yield \eqref{eq:upperbdd1} when $\ga > 0$ and \eqref{eq:upperbdd2} when $\ga < -1$. This completes our proof of parts (a) and (b) of \cref{theorem_weighted_reverse}.

\subsection{Conclusion of the proof of \cref{theorem_weighted_reverse}}
 
In \cref{sec:proofoflimitsW11} we proved parts (a) and (b) of \cref{theorem_weighted_reverse}. The lower bound for the lim\,inf in part (c) has been established in \cref{lem:liminfs}(ii), and the limiting equality for $u \in C^1_c(\R^N)$ when $p = 1$ and $-1 \leq \ga <0$ follows by combining that with the upper bound for the lim\,sup in part (ii) of \cref{lem:limsupsN>1}.
The proof of the negative result in part (c) of the theorem (generic failure for $p=1$, $-1\le \ga<0$) will be given in \cref{proposition:generic} below. \qed

\subsection{On limit formulas for $\dot{BV}(\bbR)$-functions - The proof of \cref{prop:failureoflimitforBV}}
\label{sec:BV-limit}

When $p = 1$, Poliakovsky \cite{poliakovsky} asked whether \eqref{eq:lim_weighted} still holds for $u \in \dot{BV} (\Rset^N)$ instead of $\dot{W}^{1,1}(\Rset^N)$ if $\ga = N$. More generally, one may wonder whether it is possible that for all $u \in \dot{BV}(\R^N)$, one has
\begin{equation} \label{eq:lim_weighted_BV}
\lim_{\la \to \infty} \la \nu_{\ga}(E_{\la,\ga}[u]) = \frac{\ka(1,N)}{|\ga|} \|\nabla u\|_{\cM} \quad \text{when $\ga > 0$},
\end{equation}
and
\begin{equation} \label{eq:lim_weighted_neg_BV}
\lim_{\la \to 0^+} \la \nu_{\ga}(E_{\la,\ga}[u]) = \frac{\ka(1,N)}{|\ga|} \|\nabla u\|_{\cM} \quad \text{when $\ga < 0$}.
\end{equation}
We show that this is \underline{not} the case.

First, when $-1 \leq \ga < 0$, \cref{thm:counterexamples}(i) (proved in \cref{prop:wknolimits} below) shows that even if $u \in \dot{W}^{1,1}(\R^N)$, it may happen that $\lim_{\la \to 0^+} \la \nu_{\ga}(E_{\la,\ga}[u]) = \infty$. So \eqref{eq:lim_weighted_neg_BV} cannot hold for all $u \in \dot{BV}(\R^N)$ for such $\ga$.

The following lemma provides examples of failure of \eqref{eq:lim_weighted_BV} and \eqref{eq:lim_weighted_neg_BV} when $\ga \in \R \setminus [-1,0]$, since $|\ga+1| \ne |\ga|$ unless $\ga = -1/2$:
\begin{lemma}\label{lem:failureoflimitsBV}
Suppose $N \geq 1$ and $u = \bbone_{\Omega}$ where $\Omega$ is any bounded domain in $\R^N$ with smooth boundary. Then $u \in \dot{BV}(\R^N)$ and
\[
\lim_{\la \to \infty} \la \nu_{\ga}(E_{\la,\ga}[u]) = \frac{\ka(1,N)}{|\ga+1|} \|\nabla u\|_{\cM} \quad \text{for all $\ga > -1$},
\]
while
\[
\lim_{\la \to 0^+} \la \nu_{\ga}(E_{\la,\ga}[u]) = \frac{\ka(1,N)}{|\ga+1|} \|\nabla u\|_{\cM} \quad \text{for all $\ga < -1$}.
\]
\end{lemma}
\begin{proof}
First consider the case $N = 1$. If $u = \bbone_{[0,\infty)}$, then for  every $\gamma \in \R \setminus \{-1\}$ and $\lambda > 0$, one has
\begin{equation} \label{eq:calc}
\nu_{\gamma} (E_{\la,\ga}[u]) 
= 2 \nu_{\gamma}(\{(x,y) \in \R \colon x \geq 0, y < 0, |x-y|^{-(\gamma+1)} \geq \lambda \})
= \frac{2}{|\gamma+1|} \frac{1}{\lambda},
\end{equation}
which follows from a change of variables $s = x-y$, $t = x+y$: when $\gamma > -1$, one has
\[
\nu_{\gamma} (E_{\la,\ga}[u]) = \int_0^{\lambda^{-\frac{1}{\gamma+1}}} \int_{-s}^s \dif t \,s^{\gamma-1}  \dif s = 2 \int_0^{\lambda^{-\frac{1}{\gamma+1}}} s^{\gamma} \dif s = \frac{2}{\gamma+1} \frac{1}{\lambda}
\]
while when $\gamma < -1$, one has
\[
\nu_{\gamma} (E_{\la,\ga}[u]) = \int_{\lambda^{-\frac{1}{\gamma+1}}}^{\infty} \int_{-s}^s  \dif t \,s^{\gamma-1} \dif s = 2  \int_{\lambda^{-\frac{1}{\gamma+1}}}^{\infty} s^{\gamma} \dif s = \frac{2}{|\gamma+1|}\frac{1}{\lambda}. 
\] 
A similar calculation shows that if $u = \bbone_{I_1} + \dots + \bbone_{I_j}$ is a sum of characteristic functions of finitely many open intervals whose closures are pairwise disjoint, then 
\begin{equation} \label{eq:1d_pos}
\lim_{\la \to \infty} \la \nu_{\ga}(E_{\la,\ga}[u]) = \frac{2}{|\ga+1|} \|u'\|_{\cM(\R)} \quad \text{for all $\ga > -1$},
\end{equation}
while
\begin{equation} \label{eq:1d_neg}
\lim_{\la \to 0^+} \la \nu_{\ga}(E_{\la,\ga}[u]) = \frac{2}{|\ga+1|} \|u'\|_{\cM(\R)} \quad \text{for all $\ga < -1$};
\end{equation}
we also have
\begin{equation} \label{eq:1d_bdd}
\sup_{\la > 0} \la \nu_{\ga}(E_{\la,\ga}[u]) \leq \frac{2}{|\ga+1|} \|u'\|_{\cM(\R)} \quad \text{for all $\ga \in \R \setminus\{-1\}$}.
\end{equation}

Now consider the case $N \ge 2$. Let $\Omega$ be a bounded domain in $\R^N$ with smooth boundary and $u = \bbone_{\Omega}$. Then $u \in \dot{BV}(\R^N)$ with $\|\nabla u\|_{\cM} = \cL^{N-1}(\partial \Omega)$. The method of rotation shows 
\[
\la \nu_{\gamma} (E_{\la,\ga}[u]) 
= \frac{1}{2} \int_{\Sset^{N-1}} \int_{\omega^{\perp}} \la \nu_{\gamma} (E_{\lambda, \ga}[u_{\omega, x'}]) \dif x' \dif \omega
\]
where $u_{\omega,x'}(t) \defeq u(x'+t\omega)$ for $\omega \in \S^{N-1}$ and $x' \in \omega^{\perp}$. \eqref{eq:1d_pos}, \eqref{eq:1d_bdd} and the dominated convergence theorem allows one to show that
\[
\lim_{\la \to \infty} \la \nu_{\ga}(E_{\la,\ga}[u]) = \frac{1}{|\ga+1|} \int_{\S^{N-1}} \int_{\omega^{\perp}} \|u_{\omega,x'}'\|_{\cM(\R)} \dif x' \dif \omega \quad \text{for all $\ga > -1$},
\]
and using \eqref{eq:1d_neg} in place of \eqref{eq:1d_pos} we obtain the same conclusion with $\lim_{\la \to \infty}$ replaced by $\lim_{\la \to 0^+}$ if $\ga < -1$. It remains to observe that
\begin{equation} \label{eq:Crofton}
\int_{\S^{N-1}} \int_{\omega^{\perp}} \|u_{\omega,x'}'\|_{\cM(\R)} \dif x' \dif \omega = \ka(1,N) \|\nabla u\|_{\cM};
\end{equation}
this equality holds by Fubini's theorem if $u = \bbone_{\Omega}$ is replaced by $u_{\varepsilon} := u * \rho_{\varepsilon}$ where $\rho_{\varepsilon}$ is a suitable family of mollifiers, because the left hand side is then just
\[
\int_{\S^{N-1}} \int_{\omega^{\perp}} \int_{\R} \Big|\frac{d}{dt} u_{\varepsilon}(x'+t \omega)\Big| \dif t \dif x' \dif \omega = \int_{\S^{N-1}} \int_{\R^N} |\omega \cdot \nabla u_{\varepsilon}(x)| \dif x \dif \omega
\] 
which equals $\ka(1,N) \|\nabla u_{\varepsilon}\|_{L^1(\R^N)}$. One then just need to let $\varepsilon \to 0$ to obtain \eqref{eq:Crofton}. (See also the integral-geometric formula for the surface measure in \cite{Federer}*{Chapters 2.10.15, 3.2.13 \& 3.2.26}, which extends the classical Crofton formula.)
\end{proof}

\section{From weak type bounds on quotients to $\dot W^{1,p}$ and $\dot {BV}$}\label{sec:backtoSobolev}

In this section we complete the proofs of \cref{theorem_weighted_pgr1,theorem_weighted_pequal1}  proving part (ii) of these theorems. The key tool is the BBM formula, proved in \cite{Bourgain_Brezis_Mironescu_2001} (see also \cite{Davila} for additional information for the $BV$ case). The formula is quite flexible, involving a bounded smooth domain $\Omega$ and a sequence of non-negative radial mollifiers $\rho_n(|x|)$ with $\int_0^1 \rho_n(r) r^{N-1} dr = 1$ and $\lim_{n \to \infty} \int_{\delta}^{\infty} \rho_n(r) r^{N-1} dr = 0$ for every $\delta > 0$; we will apply it in the case when $\Omega = B_R$, the ball of radius $R$ centered at $0$, and $\rho_n(r) = s_n p (2R)^{-s_n p} |r|^{-N+s_n p} \bbone_{[0,2R]}(r)$ where $\{s_n\}$ is a sequence of positive numbers tending to $0$. As a result, we conclude that if $R > 0$, $1 \leq p < \infty$, $u \in L^p(B_R)$ and 
\[
\liminf_{s \to 0^+} s \iint_{B_R \times B_R} \frac{|u(x)-u(y)|^p}{|x-y|^{N+p-s p}} \dif x \dif y < \infty,
\]
then for $p = 1$ we have $u \in \dot{BV}(B_R)$ with $\|\nabla u\|_{\mathcal{M}(B_R)}$ being bounded by $\kappa(1,N)$ times the above liminf, and for $1 < p < \infty$ we have $u \in \dot{W}^{1,p}(B_R)$ and $\|\nabla u\|_{L^p(B_R)}$ being bounded by $\kappa(p,N)/p$ times the above liminf. 

Suppose now $N \geq 1$, $1 \leq p < \infty$, $\gamma \in \R$, $u \in L^1_{\loc}(\R^N)$ and $Q_{\gamma/p} u \in L^{p,\infty}(\R^{2N},\nu_{\gamma})$. Let 
\begin{equation} \label{eq:BBMcond}
A\coloneqq \sup_{R > 0} \liminf_{s \to 0^+} s \iint_{B_R \times B_R} \frac{|u(x)-u(y)|^p}{|x-y|^{N+p-sp}} \dif x \dif y.
\end{equation}
Suppose $A$ is finite. If $p = 1$, then the BBM formula above implies $u \in \dot{BV}(B_R)$ for every $R > 0$, with $\|\nabla u\|_{\mathcal{M}(B_R)} \leq C_N A$ where $C_N$ is a dimensional constant independent of $R$; as a result, $u \in \dot{BV}(\R^N)$ (with $\|\nabla u\|_{\mathcal{M}(\R^N)} \leq C_N A$). Similarly, if $1 < p < \infty$, the BBM formula implies $u \in \dot{W}^{1,p}(\R^N)$, with $\|\nabla u\|_{L^p(\R^N)} \leq C_{N,p} A^{1/p}$; this follows from our above formulation of BBM if $u$ is additionally $L^p_{\loc}(\R^N)$, but if not, for any given $R > 0$, one can always extend $u$ by zero outside $B_R$ and denoting $u_n$ its convolution with $n \phi(n x)$ where $\phi \in C^{\infty}_c(B_1)$ has integral 1. In that case we have $u_n \in L^p(B_R)$, so the above formulation of BBM applies, and $\|\nabla u_n\|_{L^p(B_R)}$ is bounded uniformly independent of $n$ (because Minkowski's inequality implies 
\[ 
\iint_{B_R \times B_R} \frac{|u_n(x)-u_n(y)|^p}{|x-y|^{N+p-s p}} \dif x \dif y \leq \iint_{B_{R+1} \times B_{R+1}} \frac{|u(x)-u(y)|^p}{|x-y|^{N+p-s p}} \dif x \dif y 
\] 
for every $n$). Thus a subsequence of $\{\nabla u_n\}$ converges weakly in $L^p(B_R)$ to a distributional gradient $\nabla u$ of $u$ on $B_R$, and a desired bound on $\|\nabla u\|_{L^p(B_R)}$ follows for every $R>0$.

So it remains to prove that $A < \infty$. By considering truncations of $u$ we may assume additionally that $u \in L^{\infty}(\R^N)$; the reduction is based  on  the pointwise bound
\[
Q_{ \ga /p}u_n(x,y) \le Q_{\ga/ p}u(x,y)  \text{ where } u_n(x)= 
\begin{cases}
u(x) &\text{ if } |u(x)|<n,\\
n \tfrac{u(x)}{|u(x)|}&\text{ if } |u(x)|\ge n.
\end{cases} 
\] 
Using the definition of weak derivative we see by a limiting argument that the conclusion $\sup_n\|\nabla u_n\|_p\le C$ implies $\|\nabla u\|_p\le C$ if $p>1$ and $\sup_n\|\nabla u_n\|_{\cM} \le C$ implies $\|\nabla u\|_\cM\le C$.

In order to establish our estimate for bounded functions we will use Lorentz duality in the following form: if $F$, $G$ are measurable functions on $\R^{2N}$, then for any $1 < q < \infty$, we have
\begin{equation} \label{eq:Lor_dual}
\iint_{\R^N \times \R^N} F(x,y) G(x,y) \dif \nu_{\gamma} \leq q' [F]_{L^{q,\infty}(\R^{2N},\nu_{\gamma})} [G]_{L^{q',1}(\R^{2N},\nu_{\gamma})}
\end{equation}
where 
\[
[F]_{L^{q,\infty}(\R^{2N},\nu_{\gamma})} \coloneqq \sup_{\lambda > 0} \lambda \nu_{\gamma}(\{|F| > \lambda\})^{1/q} = \sup_{t > 0} t^{1/q} F^*(t),
\]
and 
\[
[G]_{L^{q',1}(\R^{2N},\nu_{\gamma})} \coloneqq \int_0^\infty \nu_{\gamma}( \{|G| > \lambda\})^{1/q'} \dif\lambda = \frac{1}{q'} \int_0^{\infty} t^{1/q'} G^{*}(t) \frac{\dif t}{t};
\]
here $F^*(t) \coloneqq \inf\{s > 0 \colon \nu_{\gamma}(\{|F| > \lambda\}) \leq s\}$ is the non-increasing rearrangement of $F$, and similarly for $G^{*}(t)$ (see \cite{Hunt_1966,stein-weiss}). Indeed, \eqref{eq:Lor_dual} follows by noticing that 
\[
\iint_{\R^N \times \R^N} F(x,y) G(x,y) \dif \nu_{\gamma} \leq \int_0^{\infty} F^*(t) G^*(t) \dif t = \int_0^{\infty} [t^{1/q} F^*(t)] [t^{1/q'} G^*(t)] \frac{\dif t}{t} 
\]
which is clearly $\leq q' [F]_{L^{q,\infty}(\R^{2N},\nu_{\gamma})} [G]_{L^{q',1}(\R^{2N},\nu_{\gamma})}$. 

First we consider the case $\gamma > 0$. For sufficiently small $s > 0$, define
\[
\theta \coloneqq \frac{s}{1+\frac{\gamma}{p}}
\]
so that $\theta \in (0,1)$ and $p-sp = p(1-\theta)(1+\frac{\gamma}{p})-\gamma$. Then for every $R > 0$,
\begin{align*}
&\iint_{B_R \times B_R} \frac{|u(x)-u(y)|^p}{|x-y|^{N+p-sp}} \dif x\dif y \\
=& \iint_{\R^N \times \R^N} \left( Q_{\gamma/p}u(x,y) \right)^{p(1-\theta)} \left( |u(x)-u(y)| \bbone_{B_R \times B_R}(x,y) \right)^{p\theta} \dif \nu_{\gamma} \\
\leq & \frac{1}{\theta} \left [\left( Q_{\gamma/p}u \right)^{p(1-\theta)} \right]_{L^{\frac{1}{1-\theta},\infty}(\R^{2N},\nu_{\gamma})} \left [ |u(x)-u(y)|^{p\theta} \right]_{L^{\frac{1}{\theta},1}(B_R \times B_R,\nu_{\gamma})}
\end{align*}
by \eqref{eq:Lor_dual}. But 
\[
\left [\left( Q_{\gamma/p}u \right)^{p(1-\theta)} \right]_{L^{\frac{1}{1-\theta},\infty}(\R^{2N},\nu_{\gamma})} 
= \left [Q_{\gamma/p}u\right]_{L^{p,\infty}(\R^{2N},\nu_{\gamma})}^{p(1-\theta)}
\]
and
\begin{align*}
\left [ |u(x)-u(y)|^{p\theta} \right]_{L^{\frac{1}{\theta},1}(B_R \times B_R,\nu_{\gamma})} &\leq (2\|u\|_{L^{\infty}(\R^N)})^{p\theta} [\bbone_{B_R \times B_R}]_{L^{\frac{1}{\theta},1}(\R^N \times \R^N,\nu_{\gamma})} \\
&= (2\|u\|_{L^{\infty}(\R^N)})^{p\theta} \nu_{\gamma}(B_R \times B_R)^{\theta}, 
\end{align*}
from which it follows that 
\[
s \iint_{B_R \times B_R} \frac{|u(x)-u(y)|^p}{|x-y|^{N+p-sp}} \dif x \dif y \leq \frac{s}{\theta} \left [Q_{\gamma/p}u\right]_{L^{p,\infty}(\R^{2N},\nu_{\gamma})}^{p(1-\theta)} (2\|u\|_{L^{\infty}(\R^N)})^{p\theta} \nu_{\gamma}(B_R \times B_R)^{\theta}.
\]
Furthermore, since $\gamma > 0$, we have
\[
\nu_{\gamma}(B_R \times B_R) \leq |B_R| \int_{B_{2R}} \frac{1}{|h|^{N-\gamma}} \dif h < \infty.
\]
Recall $\theta = \frac{s}{1+\frac{\gamma}{p}}$. Thus as $s \to 0^+$, we have
\[
\limsup_{s \to 0^+} s \iint_{B_R \times B_R} \frac{|u(x)-u(y)|^p}{|x-y|^{N+p-sp}} \dif x \dif y \leq \left( 1+\frac{\gamma}{p} \right) \left [Q_{\gamma/p}u\right]_{L^{p,\infty}(\R^{2N},\nu_{\gamma})}^p < \infty.
\]
Since this upper bound holds uniformly over all $R > 0$, this concludes the argument for the case $\gamma > 0$.

Next we turn to the case $\gamma \leq 0$. We then observe that for $0 < s < 1$ and every $R > 0$,
\begin{align*}
& \iint_{B_R \times B_R} \frac{|u(x)-u(y)|^p}{|x-y|^{N+p-sp}} \dif x \dif y \\
=& \iint_{\R^N \times \R^N} \left( Q_{\gamma/p}u(x,y) \right)^{p(1-\frac{s}{2})} \left( |u(x)-u(y)| |x-y|^{1-\frac{\gamma}{p}} \bbone_{B_R \times B_R} \right)^{p\frac{s}{2}} \dif \nu_{\gamma} \\
\leq & \frac{2}{s} \left[ \left(Q_{\gamma/p}u\right)^{p(1-\frac{s}{2})} \right]_{L^{\frac{1}{1-\frac{s}{2}},\infty}(\R^{2N},\nu_{\gamma})} \left[ \left( |u(x)-u(y)| |x-y|^{1-\frac{\gamma}{p}} \right)^{p\frac{s}{2}} \right]_{L^{\frac{2}{s},1}(B_R \times B_R,\nu_{\gamma})}.
\end{align*}
Again
\[
\left[ \left(Q_{\gamma/p}u\right)^{p(1-\frac{s}{2})} \right]_{L^{\frac{1}{1-\frac{s}{2}},\infty}(\R^{2N},\nu_{\gamma})} = \left[ Q_{\gamma/p}u\right]_{L^{p,\infty}(\R^{2N},\nu_{\gamma})}^{p(1-\frac{s}{2})}
\]
and 
\begin{multline}
\left[ \left( |u(x)-u(y)| |x-y|^{1-\frac{\gamma}{p}} \right)^{p\frac{s}{2}} \right]_{L^{\frac{2}{s},1}(B_R \times B_R,\nu_{\gamma})} \\ \leq (2 \|u\|_{L^{\infty}(\R^N)})^{p \frac{s}{2}} \left[ |x-y|^{(p-\gamma)\frac{s}{2}} \right]_{L^{\frac{2}{s},1}(B_R \times B_R,\nu_{\gamma})}.
\end{multline}
We will show that
\begin{equation} \label{eq:Lor_est_2s}
\limsup_{s \to 0^+} \left[ |x-y|^{(p-\gamma)\frac{s}{2}} \right]_{L^{\frac{2}{s},1}(B_R \times B_R,\nu_{\gamma})} \leq 1 - \frac{\gamma}{p}
\end{equation}
when $\gamma \leq 0$. We then see that
\[
\limsup_{s \to 0^+} s \iint_{B_R \times B_R} \frac{|u(x)-u(y)|^p}{|x-y|^{N+p-sp}} \dif x\dif y \leq 2 \left(1-\frac{\gamma}{p}\right) \left[ Q_{\gamma/p}u\right]_{L^{p,\infty}(\R^{2N},\nu_{\gamma})}^{p}
\]
which concludes the argument in this case since this bound is uniform in $R > 0$.

It remains to prove \eqref{eq:Lor_est_2s} when $\gamma \leq 0$. Note that in this case $p - \gamma > 0$, so $|x-y|^{(p-\gamma)\frac{s}{2}} \leq (2R)^{(p-\gamma)\frac{s}{2}}$ on $B_R \times B_R$. Thus 
\[
\left[ |x-y|^{(p-\gamma)\frac{s}{2}} \right]_{L^{\frac{2}{s},1}(B_R \times B_R,\nu_{\gamma})}
= \int_0^{(2R)^{(p-\gamma)\frac{s}{2}}} \nu_{\gamma}\{(x,y) \in B_R \times B_R \colon |x-y|^{(p-\gamma)\frac{s}{2}} > \lambda \}^{\frac{s}{2}} \dif \lambda 
\]
If $\gamma < 0$, then
\[
\nu_{\gamma}\{(x,y) \in B_R \times B_R \colon |x-y|^{(p-\gamma)\frac{s}{2}} > \lambda \}
\leq |B_R| \int_{|h| > \lambda^{\frac{2}{s(p-\gamma)}}} \frac{1}{|h|^{N-\gamma}} \dif h
\leq \sigma_{N-1} |B_R| \frac{1}{|\gamma|} \lambda^{ \frac{2\gamma}{s(p-\gamma)}}
\]
where $\sigma_{N-1}$ is the surface area of $\Sset^{N-1}$. Hence in this case,
\begin{align*}
\left[ |x-y|^{(p-\gamma)\frac{s}{2}} \right]_{L^{\frac{2}{s},1}(B_R \times B_R,\nu_{\gamma})}
&\leq \left(  \sigma_{N-1} |B_R| \frac{1}{|\gamma|} \right)^{\frac{s}{2}} \int_0^{(2R)^{(p-\gamma)\frac{s}{2}}}  \lambda^{ \frac{\gamma}{p-\gamma}} \dif \lambda \\
&= \left(1-\frac{\gamma}{p}\right) \left(  \sigma_{N-1} |B_R| \frac{1}{|\gamma|} \right)^{\frac{s}{2}} (2R)^{p\frac{s}{2}}.
\end{align*}
(Here we used $\frac{\gamma}{p-\gamma} = -\frac{1}{1-\frac{\gamma}{p}} \in (-1,0)$ whenever $\gamma < 0$.)
This proves \eqref{eq:Lor_est_2s} when $\gamma < 0$.  Next, suppose $\gamma = 0$. Then 
\begin{align*}
\left[ |x-y|^{(p-\gamma)\frac{s}{2}} \right]_{L^{\frac{2}{s},1}(B_R \times B_R,\nu_{\gamma})}
&= \int_0^{(2R)^{p\frac{s}{2}}} \nu_{0}\{(x,y) \in B_R \times B_R \colon |x-y|^{p\frac{s}{2}} > \lambda \}^{\frac{s}{2}} \dif \lambda \\
&\leq \int_0^{(2R)^{p\frac{s}{2}}} \left( |B_R| \int_{\lambda^{\frac{2}{sp}} \leq |h| \leq 2R} \frac{1}{|h|^{N}} dh \right)^{\frac{s}{2}} \dif \lambda \\
&= \int_0^{(2R)^{p\frac{s}{2}}} \left( |B_R| \omega_{N-1} \frac{2}{ps} \log\left( \frac{(2R)^{p\frac{s}{2}}}{\lambda} \right) \right)^{\frac{s}{2}} \dif \lambda \\
&= (2R)^{p\frac{s}{2}} \int_0^{1} \left( |B_R| \omega_{N-1} \frac{2}{ps} \log\left( \frac{1}{\lambda} \right) \right)^{\frac{s}{2}} \dif \lambda
\end{align*}
which shows \eqref{eq:Lor_est_2s} remains valid when $\gamma = 0$ by the dominated convergence theorem. \qed

\section{ Finiteness of $\nu_0(E_{\la,0}[u]) $ and the Lipschitz norm} \label{sec:gamma=0}
In this section we prove \cref{thm:counterexamples-gamma=0} which we put in the following more precise form.
\begin{proposition} \label{thm:gamma=0} 
Let $u$ be  locally integrable on $\bbR^N$ and  $\nabla u\in L^1_\loc(\bbR^N)$.
Then 
\[ 
\nu_0( E_{\la,0} [u])= 
\begin{cases}  0 &\text{ if } \la> \|\nabla u\|_\infty, \\ 
\infty &\text{ if } \la< \|\nabla u\|_\infty. 
\end{cases} 
\]
\end{proposition}
\begin{proof} 
First assume  $\nabla u\in L^\infty$ and  $\la>\|\nabla u\|_\infty$. Then for  every $h\in \bbR^N$ we have $\frac{|u(x+h)-u(x)|}{|h|}\le \la $ for almost every $x\in \bbR^N$.  This immediately implies  $\nu_0(E_{\la,0}[u])=0$.

For the more substantial part  assume $\la< \|\nabla u\|_\infty$ where  $\|\nabla u\|_\infty$ may be finite or infinite. We need to show that $\nu_0(E_{\la,0}[u])=\infty$. 
We pick $\la_1, \la_2$ such that
\[ 
\la<\la_1<\la_2<\|\nabla u\|_\infty. 
\]
Let $B_R=\{x\in \bbR^N:|x|<R\}$ and assume that $R>1$ is so large that $\norm{\nabla u}_{L^\infty (B_R)} > \lambda_2$ is not the zero distribution on $B_R$.  Let $\chi\in C^\infty_c$ such that $\chi(x)=1$ in a neighborhood of $\overline{B_{2R}}$ and set $u_\circ=\chi u$, Then $\nabla u_\circ=\nabla u$  in the sense of $L^1(B_{2R})$.  There is a measurable set $F_0\subset B_R$ of positive measure  such that $|\nabla u(x)|> \la_2$ for all $x\in F_0$.

Fix  $0<\eps\ll  1-\frac{\la_1}{\la_2} $.  We now consider the set $\fS_\eps$ of all spherical balls  $S \subset \Sset^{N-1}$ with positive radius and the property that $\inn{\theta_1}{\theta_2} >1-\eps$ for all $\theta_1, \theta_2\in S$.  By pigeonholing there exists a spherical ball  $S\in \fS_\eps$  and a  Lebesgue measurable subset $F\subset F_0$  such that $\cL^N(F)>0$ and $\frac{\nabla u(x)}{|\nabla u(x)|} \in S$ for all $x\in F$. For the remainder of the argument we fix this spherical ball $S$; we denote by $\sigma(S)$ its spherical measure.

We first note that  for $|h|\le 1$  and for almost every $|x|\le R$ 
\Be \label{eq:preliminary-obs}
\frac{u (x+h)- u(x)}{|h|} =
\frac{u_\circ (x+h)- u_\circ(x)}{|h|} 
= \ \Biginn{\frac{h}{|h|} }{\int_0^1 \nabla u_\circ (x+sh) \dif s}.
\Ee
Secondly since the translation operator is continuous in the strong operator topology of $L^1$ we see that there exists $\delta_0<1$ such that 
\Be \label{eq:continuity}  \|\nabla u_\circ (\cdot + w)-\nabla u_\circ \|_{L^1(\bbR^N)} < \frac{ \cL^N(F) (\la_1-\la) }{10} \text{  for  } |w|\le \delta_0.
\Ee

In what follows we let $\delta\ll \delta_0$ and set
\[ 
S(\delta,\delta_0)=\Bigl\{h\in \bbR^N: \delta\le |h|\le \delta_0, \frac{h}{|h|} \in S \Bigr\}. 
\]
Let 
\[
\cE_0= \Bigl\{ (x,h): x\in F, \, h\in S(\delta,\delta_0), \,\frac{|u(x+h)-u(x) |}{|h|} >\la \Bigr\}
\]
so that $(x, h) \in \cE_0$ implies $(x, x+h) \in  E_{\la,0}[u]$.
We then have by \eqref{eq:preliminary-obs}
\begin{align} \notag
    \nu_0(E_{\la,0}[u])&\ge \nu_0 (\cE_0) 
    = \nu_0\Bigl(\Bigl\{(x,h): x\in F, \, h\in S(\delta,\delta_0), \, |\inn { \tfrac{h}{|h|}}{ \int_0^1 \nabla u_\circ (x+sh)\dif s} | >\la  \Bigr\}\Bigr) \\ 
 \label{eq:difference}
&\ge \nu_0 (\cE_1) -\nu_0 (\cE_2) 
\end{align}
where
\begin{align*}
    \cE_1&=\{ (x,h): x\in F,\, h\in S(\delta,\delta_0) ,\, |\inn{\tfrac{h}{|h|}}{\nabla u_\circ(x)} |> \lambda_1\}   
    \\
    \cE_2&=\Bigl\{(x,h): x\in F, h\in S(\delta,\delta_0) ,\, \int_0^1 |\nabla u_\circ(x+sh)-\nabla u_\circ(x) | \dif s >\lambda_1-\la  \Bigr\}.
\end{align*}
Indeed, if $(x,h) \notin \cE_0\cup  \cE_2$ then 
\[|\inn{\tfrac{h}{|h|}}{\nabla u_\circ(x)} |
\le | \inn { \tfrac{h}{|h|}}{ \int_0^1 \nabla u_\circ (x+sh)\dif s}|+ 
 \int_0^1 |\nabla u_\circ(x+sh)-\nabla u_\circ(x) | \dif s\]
 which is then $\le \la_1$, so $(x,h)\notin \cE_1$, establishing $\cE_1\subset \cE_0\cup\cE_2$ and thus \eqref{eq:difference}.

The set $\cE_1$  does not change if we  replace $u_\circ$ by $u$  in its definition. Since 
\[ 
\inn{\tfrac{h}{|h|} }{\nabla u(x)} \ge  (1-\eps)  |\nabla u(x)| > (1-\eps)\la_2>\la_1 \text{ for $x\in F$, $\tfrac{h}{|h|}\in S$ } 
\] 
we get 
\[
\nu_0(\cE_1) \ge \int_F dx \int_{S(\delta,\delta_0)} \frac{\dif h}{|h|^N} =
\cL^N(F)\sigma(S) \log (\delta_0/\delta).
\]

Moreover, using \eqref{eq:continuity} and Chebyshev's inequality we see that 
\begin{align*} 
\nu_0(\cE_2) &\le \int_{S(\delta,\delta_0)} \frac{ \int_0^1 \|\nabla u_\circ(\cdot+sh)-\nabla u_\circ \|_{L^1(\bbR^N)} \dif s }{\la_1-\la} \frac{dh}{|h|^N} 
\\
&\le \int_{S(\delta,\delta_0)}  \frac {\cL^N(F)(\la_1-\la)/10}{\la_1-\la}  \frac{dh}{|h|^N}  = \frac{\cL^N(F)}{10}\sigma(S) \log(\delta_0/\delta) 
\end{align*}
and hence putting pieces together we obtain for $\delta<\delta_0$
\[
\nu_0(E_{\la,0}[u]) \ge \nu_0(\cE_1)-\nu_0(\cE_2)>  \frac { \cL^N(F)}{2}\sigma(S)  \log(\delta_0/\delta).
\]
Here $\delta<\delta_0$ was arbitrary and by letting $\delta\to 0$ we conclude that $\nu_0(E_{\la,0}[u]) =\infty$.
\end{proof}

We now give a more precise version of \cref{lem:domains}.
\begin{lemma} \label{lem:domainsprecise} 
Let $\Omega \subset \R^N$ be a bounded domain with Lipschitz boundary and let $u=\bbone_\Omega$. Then $u\in{BV}(\bbR^N) \setminus \dot W^{1,1}(\bbR^N)$  with
\[
  \nu_0(E_{\la,0}[u])\le C_\Omega \times 
\begin{cases} 
   \log(2/\la) &\text{ if } \la\le 1,  \\
   \la^{-1} &\text{ if } \la>1;
\end{cases}
\] 
in particular we have $\sup_{\la>0} \la\,  \nu_0(E_{\la,0}[u])<\infty$.
\end{lemma}
\begin{proof}
Let 
\[
E(r,\la)=\{ (x,y) \in E_{\lambda,0}[u] \colon r\le |x-y|\le 2r\}.
\]
We begin with the observation  that $r\la\le 2$ if $\nu_0(E(r,\la)) > 0$. 
Furthermore, if $(x,y)\in E(r,\la)$ for some $y \in \R^N$, then  $x$ belongs to the  $2r$-neighborhood of $\partial \Omega$. The Lebesgue measure of such a neighborhood is $O(r)$ if $r \leq r_0$ where $r_0$ is some positive constant depending on $\Omega$ (because the boundary of a bounded Lipschitz domain can be covered by finitely many Lipschitz graphs, and the $2r$-neighborhood of such graphs can be approximated by a union of $O(r)$ neighborhoods of suitable hyperplanes). 
Hence for $r\le r_0$ we have $\nu_0(E(r,\la)) \leq Cr$ if $r\le 2/\la$ and $\nu_0(E(r,\la)) = 0$ if $r>2/\la$. 
As a result, if $ 2/\la\le r_0$ we get
\[
\nu_0(E_{\la,0}[u]) \le \sum_{j \in \mathbb{Z} \colon 2^j \le 2/\lambda} \nu_0(E(2^j,\lambda)) \lc \la^{-1}
\] 
and if $2/\la>r_0$ we get 
\[\nu_0(E_{\la,0}[u]) 
\le \sum_{j \in \mathbb{Z} \colon 2^j \le r_0} \nu_0(E(2^j,\lambda)) + 2 \int_{\Omega} \int_{r_0\le |x-y|\le 2/\la} \frac{dy}{|x-y|^N} dx \lc 1+ \log(\la^{-1} ).
\]
\end{proof}

\section{When the upper bound \eqref{eq:Marc-1} fails}\label{sec:lower-bounds} 

In this section we make various constructions demonstrating the failure of \eqref{eq:Marc-1} in the range $-1\le \gamma<0$, and give the proof of   \cref{thm:counterexamples}. We first establish 
\begin{proposition} \label{proposition_Mp_gamma_-1_0_1}
Suppose $N \geq 1$ and  $-1 \le  \gamma < 0$.
\begin{enumerate}[(i)]
\item For every $m > 0$, there exists $u \in C^{\infty}_c(\R^N)$ such that
\begin{equation}\label{eq:lower1}
\nu_\ga(E_{1,\gamma}[u]) >m \norm{\nabla u}_{L^1(\R^N)}.
\end{equation}
\item  There exists $C = C(N,\gamma)>0$ and $p_0 = p_0(N,\ga) > 1$ such that for all $1 < p < p_0$,
\begin{equation} \label{eq:lowerp}
\sup\limits_{ \substack{u\in C^\infty_c(\bbR^N)\\ \norm{\nabla u}_{L^p}\le 1}} \nu_\ga( E_{1,\gamma/p}[u]) \ge 
 C \frac{p}{p-1}.
 \end{equation}
 \end{enumerate}
 \end{proposition}

\subsection{\it Proof of \cref{proposition_Mp_gamma_-1_0_1}:   The case $\gamma=-1$.}\label{sec:proof-ga=1}
Here we may choose, for $m>1$, 
\begin{equation} \label{eq:vmdef}
    v_m = \bbone_{B_1} * \eta_{1/m} \in C^{\infty}_c(\R^N)
\end{equation} 
where $\eta_{1/m}(x) := 2^{mN} \rho(2^m x)$ for some non-negative, radially decreasing $\eta \in C^{\infty}_c(B_1)$ with $\int_{\R^N} \rho = 2$.
Then when $1 \leq p < \infty$ and $m \leq p' = p/(p-1)$ (which is no restriction on $m$ if $p = 1$) we have $\|\nabla v_m\|_p \lc 2^{m/p'} \lc 1$, while $E_{1,-1/p}[v_m]  \supseteq \{|x| \leq 1-2^{-m}, 1 + 2^{-m} \leq |y| \leq 2\}$, and hence
\[
\begin{split}
\nu_{-1}(E_{1,-1/p}[v_m]) 
\geq &\int_{|x| \leq 1-2^{-m}} \int_{1 + 2^{-m} \leq |y| \leq 2} |x-y|^{-1-N} \dif x \dif y \\
\geq & c_N \int_{|x| \leq 1-2^{-m}} (1 + 2^{-m} - |x|)^{-1} - (2-|x|)^{-1} \dif x \geq c_N' m.
\end{split}
\]
This proves both (i) and (ii) of \cref{proposition_Mp_gamma_-1_0_1} in the case $\ga = -1$.
\qed

\subsection{The case $-1 < \gamma < 0$: Examples of Cantor-Lebesgue type on the real line.}\label{sec:Cantor}
We now discuss some examples related to self-similar Cantor sets of dimension $\beta=1+\gamma$. 
Recall the definition of $\nu_\ga$, $Q_\ga$ in \eqref{eq:nugamma-definition}, \eqref{Qbetadef} and observe the behavior under dilations: 
\Be\label{eq:mubeta-dil} 
\nu_{\gamma}(tE)= t^{1+\gamma}  \nu_{\gamma}(E) \,.
\Ee
We  have
\begin{lemma} \label{lem:wk} 
Let $-1<\ga<0$.
There exist   constants $c_\ga>0$,   $C_\ga>0$, and a  sequence of functions $g_m \in C^\infty(\bbR)$  with $g_m(x)=0$ for $x\leq 0$ and $g_m(x)=1$ for $x\ge 1$, such that for all $1 \leq p < \infty$,
  \Be \label{eq:derivgm} 
  \|g_m'\|_p \le c_\ga 2^{\frac{m |\ga| }{1+\ga}( 1-\frac 1p)}  
  \Ee  
 and if $m-1 \leq \frac{\ga+1}{|\ga|} \frac{p}{p-1}$, then
 \begin{align} \label{eq:Qbeta-lowerbd_p}
 \nu_{\ga}\big(\{ (x,y) \in [0,1]^2: |Q_{\ga/p} g_m (x,y) | > \tfrac{1}{4} \}\big)  \ge m/C_\ga\,.
 \end{align}
 \end{lemma}
 
 \begin{proof}
For $-1 < \ga < 0$ let \begin{equation}\label{rhodef} \rho=2^{-\frac{1}{1+\ga}} \end{equation}
so that $0<\rho<1/2$.
We construct \(g_m\) such that its derivative is supported on the \(m\)--th step of the construction of symmetric Cantor sets of dimension \(\beta=1+\ga=\frac{\log 2}{\log(1/\rho)}\), with an equal variation on each of its \(2^m\) components \cite{mattila}*{ch.\ 8.1}.

Let \(g_0 \in C^\infty (\Rset)\) be such that \(0 \le g_0 \le 1\), \(g_0 (x) = 0\) for \(x \le \rho\) and \(g_0 (x) = 1\) for \(x \ge 1 - \rho\). Set for \(m \in \Nset\),
\[
 g_{m + 1} (x) \defeq \tfrac{1}{2} g_m \bigl(\tfrac{x}{\rho}\bigr)
 + \tfrac{1}{2} g_m \bigl(1 - \tfrac{1 - x}{\rho}\bigr).
\]
Since \(\rho < 1/2\), we have for \(p \in [1,\infty)\),  $\|g_{m+1}'\|_{L^p(\bbR)}^p = 2\times (2\rho)^{-p} \rho \|g_m'\|_{L^p(\bbR)}^p$ and thus 
\[
\norm{g_m'}_{L^p (\Rset)} =  (2\rho)^{(\frac 1p-1)m} \norm{g_0'}_{L^p (\Rset)}
=2^{\frac{m|\ga|}{\ga+1}(1-\frac 1p)} \norm{g_0'}_{L^p (\Rset)}.
\] 

Fix now $1 \leq p < \infty$, and for $m \in \Nset$, $\la > 0$ define  
\[
A_{m,\la} \coloneqq \nu_{\ga} \brk{\set{(x, y) \in \intvc{0}{1}^2 :\abs{Q_{\ga/p} g_{m} (x, y)}  > \lambda}}.
\]
Our goal is to estimate $A_{m,1/4}$, which we do by deriving a recursive estimate for $A_{m,\la}$.
We have the decomposition 
\begin{equation}
\label{eq_ja5diiPohth1aLahtie3Idua}
\begin{split}
A_{m+1,\la}&\ge \, 
\nu_{\ga} \brk{\set{(x, y) \in \intvc{0}{\rho}^2:\abs{Q_{\ga/p} g_{m+1} (x, y)} > \la}}\\
&+\,
\nu_{\ga} \brk{\set{(x, y) \in \intvc{1- \rho}{1}^2:\abs{Q_{\ga/p} g_{m+1} (x, y)} > \la}}\\
&+\,
\nu_{\ga} \brk{\set{(x, y) \in \intvc{0}{\rho} \times \intvc{1-\rho}{1} :\abs{Q_{\ga/p} g_{m+1} (x, y)} > \la}}.
\end{split}
\end{equation}
Using the definition of $g_{m+1}$, \eqref{rhodef} and \eqref{eq:mubeta-dil}, we compute the first term in the right-hand side of \eqref{eq_ja5diiPohth1aLahtie3Idua} as
\begin{equation}
\label{eq_Aejai1ohchoore0eegee9uoB}
\begin{split}
 &\nu_{\ga} \brk{\set{(x, y) \in \intvc{0}{\rho}^2 : \abs{Q_{\ga} g_{m + 1} (x, y)}  > \la }}\\
 &= \nu_{\ga} \brk{\set{(\rho w, \rho z): (w,z) \in \intvc{0}{1}^2, \, \abs{Q_{\ga} g_{m} (w, z)}  >  2 \rho^{1+\frac{\ga}{p}} \la }}\\
 &= \rho^{\ga+1} \nu_{\ga} \brk{\set{(w, z) \in \intvc{0}{1}^2 :  \abs{Q_{\ga} g_{m} (w, z)}  > 2^{\frac{|\ga|}{p'(\ga+1)}}\la }} = \tfrac{1}{2} A_{m,s\la}
\end{split}
\end{equation}
where $s := 2 \rho^{1+\frac{\ga}{p}} = 2^{\frac{|\ga|}{p'(\ga+1)}}$,
and similarly the second term as
\begin{equation}
\label{eq_iphein1ieshua9Oon9Niequ1}
\begin{split}
 &\nu_{\ga} \brk{\set{(x, y) \in \intvc{1-\rho}{1 }^2  :  \abs{Q_{\ga} g_{m + 1} (x, y) } > \tfrac{1}{2} }}
  = \tfrac{1}{2} A_{m,s\la}.
 \end{split}
\end{equation}
Thus
\[
A_{m+1,\la} \geq A_{m,s\la} + \nu_{\ga} \brk{\set{(x, y) \in \intvc{0}{\rho} \times \intvc{1-\rho}{1} :\abs{Q_{\ga/p} g_{m+1} (x, y)} > \la}},
\]
which iterates to give
\[
A_{m,1/4} \ge A_{0,s^m/4} + \sum_{j=1}^{m}  \nu_{\ga} \brk{\set{(x, y) \in \intvc{0}{\rho} \times \intvc{1-\rho}{1} :\abs{Q_{\ga/p} g_{j} (x, y)} > \tfrac{1}{4} s^{m-j}  }}.
\]
We drop the first term, and note that as long as $m-1 \leq \frac{\ga+1}{|\ga|} \frac{p}{p-1}$, we have $\frac{1}{4} s^{m-j} 
\leq \frac{1}{2}$ for all $j = 1, \dots, m$. Moreover, for every $x \in [0,\rho^2] \times [1-\rho^2,1]$ and every $j \geq 1$, we have $g_j(x) \leq 1/4$ and $g_j(y) \geq 3/4$, so $\abs{Q_{\ga/p} g_{j} (x, y)} > \frac{1}{2}$. Thus we obtain the desired conclusion
\[
A_{m,1/4} \geq m \nu_{\ga}([0,\rho^2] \times [1-\rho^2,1]) = m / C_{\ga}.
\]
\end{proof}

\subsection{Conclusion of the proof of  \cref{proposition_Mp_gamma_-1_0_1}.}
\label{sec:counterex-higherdim}

We continue with the case $-1<\gamma<0$.
 Let $\eta_1\in C_c^\infty(\bbR)$ supported in $(-1,2)$ such that $\eta_1(s)=1$ on $(-1/2,3/2)$ and $0 \le \eta_1(s) \le 1$ for all $s\in \bbR$.

We split $x=(x_1,x')$ with $x'\in \bbR^{N-1}$,  where the variable $x'$ should simply be dropped in the case $N=1$. Set $\eta(x)=\prod_{i=1}^N \eta_1(x_i)$  and define
\Be\label{eq:umdef} u_m(x_1,x') = 16 g_m(x_1)\eta(x)\Ee where $g_m$ is as in \cref{lem:wk}.
Then $u_m \in C^{\infty}_c(\R^N)$, and if $1 \leq p < \infty$ and $m-1 \leq \frac{\ga+1}{|\ga|} \frac{p}{p-1}$ we have $\|\nabla u_m\|_p \lesssim 1$. Both parts of \cref{proposition_Mp_gamma_-1_0_1} will follow, if we can prove that under the same hypotheses on $p$ and $m$, we have
\begin{equation} \label{lower-bound-N}
\nu_{\ga}(E_{1,\ga/p}[u_m]) \geq c(N,\ga) m - C(N,\ga)^p.
\end{equation}
We aim to reduce to the one-dimensional situation in  \cref{lem:wk} and split 
\[
Q_{\ga/p}u_m(x,y) =16\eta(x) \frac{g_m(x_1) -g_m(y_1) }{|x-y|^{1+\frac{\ga}{p}} } + 16 g_m(y_1) \frac{ \eta(x)-\eta(y)}{|x-y|^{1+\frac{\ga}{p}} } = I_m(x,y)+I\!I_m(x,y)
\]
so that 
\begin{align}\notag
\nu_\ga(E_{1,\ga/p}[u_m])  
&\ge \iint\limits_{\substack {x_1,y_1 \in [0,1] \\ |I_m(x,y)+I\!I_m(x,y)| > 1}}|x-y|^{\ga-N} \dif x \dif y 
\\
\label{two-terms} 
&\ge \iint\limits_{\substack{x \in [0,1]^N, y_1 \in [0,1] \\ |x_1-y_1|\ge |x'-y'| \\ |I_m(x,y)|> 2}}|x-y|^{\ga-N} \dif x \dif y 
-\iint\limits_{\substack{|I\!I_m(x,y)| >  1} }|x-y|^{\ga-N} \dif x \dif y .
\end{align} 

Clearly if $B_2$ is the ball in $\bbR^N$ of radius $2$ centered at the origin then   \[|I\!I_m(x,y)|\leq c_N |x-y|^{-\frac{\ga}{p}} (\bbone_{B_2}(x)+\bbone_{B_2}(y)) \] and it follows immediately (since $-\ga>0$) that 
\[ 
\iint\limits_{\substack{|I\!I_m(x,y)|> 1}}|x-y|^{\ga-N} \dif x \dif y  \le |\ga|^{-1} C(N)^p.
\]

For the first term in \eqref{two-terms}, we prove a lower bound and estimate by integrating in $y'$
\begin{align*}
\iint\limits_{\substack{x \in [0,1]^N, y_1 \in [0,1] \\ |x_1-y_1|\ge |x'-y'| \\|I_m(x,y)|> 2}}|x-y|^{\ga-N} \dif x \dif y 
&\ge 
\iint\limits_{\substack{x \in [0,1]^N, y_1 \in [0,1] \\ |x_1-y_1|\ge |x'-y'| \\ 
\frac{|16 g_m(x_1)- 16 g_m(y_1)|} {|x_1-y_1|^{1+\frac{\ga}{p}}}> 4}}|x-y|^{\ga-N} \dif x \dif y 
\\
\\&\ge c_N 
\iint\limits_{\substack{x_1,y_1 \in [0,1] \\ 
|Q_{\ga/p}g_m(x_1,y_1)| > \tfrac 14}}|x_1-y_1|^{\ga-1} \dif x_1 \dif y_1,
\end{align*} 
but by \cref{lem:wk} the last expression is bounded below for large $m$  by $c_N m/C_{\ga}$ under our hypothesis on $m$. 
This concludes the proof of \eqref{lower-bound-N}. 
\qed

For later purposes, note the inequality \eqref{two-terms} (with $p = 1$) and the argument that follows proved also that for all sufficiently large $m > m(N,\ga)$, we have
\Be \label{eq:lowbd[0,1]}
\nu_{\ga}(E_{1,\ga}[u_m] \cap ([0,1] \times \R^{N-1})^2) \geq c(N,\ga) m.
\Ee

\subsection{Examples related to Theorems \ref{theorem_weighted_reverse} and \ref{thm:counterexamples}}
We now consider the limit \eqref{eq:lim_weighted_neg} in the range $-1\le \ga<0$ and provide counterexamples for cases where $u$ is no longer required to be  a $C^\infty_c$ function. 
The following proposition covers part (i) of \cref{thm:counterexamples}.
\begin{proposition} \label{prop:wknolimits} 
Let $-1\le \gamma<0$. Let $s\mapsto \omega(s)$ be any decreasing function on $[0,\infty)$ with $\omega(0)\le 1$ and $\omega(s)>0$ for all $s\ge 0$.
Then there exists a $C^\infty$ function $u\in \dot W^{1,1}(\bbR^N)$ such that 
\Be\label{eq:log-decay} 
|u(x)|\le C \omega(|x|) \text{ for all $x\in \bbR^N$ }
\Ee
and \Be \label{zero-limit}
\lim_{\la\searrow 0} \la \nu_\gamma(E_{\la,\gamma}[u]) =\infty .
\Ee
\end{proposition} 

 \begin{proof}
 We consider the case $-1<\ga<0$.  
 Let $u_m \in C^{\infty}_c(\R^N)$ be as in \eqref{eq:umdef} and define 
\begin{equation}\label{eq:fmdef} 
f_m(x)= u_m(x_1-2,x') 
\Ee 
so that $f_m(x)=0$ if $x_1\notin [1,4]$.
   Let, for  $n\in \bbN$, 
   \Be \label{eq:choice-of-global-parameters} 
   \text{ $R_n=2^{2n}$, \ $\la_n= R_n^{-(N+\gamma)} \omega(R_{n+1})$, \ $m(n)\ge 4\frac{\la_n}{\la_{n+1}} \omega(R_{n+1})^{-1} n^3$.}
   \Ee
  We also assume $m(n) > m(N,\ga)$ so that by \eqref{eq:lowbd[0,1]} in \cref{sec:counterex-higherdim},
   \Be
   \label{Qbeta-fm-bd}
   \nu_\gamma(\{(x,y): x_1,y_1 \in [2,3], | Q_{\gamma} f_{m(n)} (x,y)|> 1 \}) \ge c(N,\ga) m(n)
   \Ee
   for all $n \in \Nset$.
   Finally let 
 \Be\label{eq:def-of-u} 
 u(x) =\sum_{n=2}^\infty \frac{\omega(R_{n+1}) }{R_n^{N-1}n^2} f_{m(n)}\Bigl(\frac{x}{R_n}\Bigr).
 \Ee
  Since $\|f_{m}\|_{\dot W^{1,1} } \le C$, and $\omega$ is bounded,
    it is easy to see that the sum converges in $\dot W^{1,1} (\Rset^N)$, and that $u$ is in  $\dot W^{1,1}(\Rset^N)$. Also,  the supports of $f_{m(n)}(R_n^{-1} \cdot)$, namely $[R_n, 4R_n] \times [-4R_n, 4R_n]^{N-1}$, are disjoint as $n$ varies, so clearly $u\in C^\infty(\bbR^N)$.
  Since    $\|f_{m}\|_{L^\infty} \le C$,  we have 
  \[
  |u(x)|\le \om(R_{n+1}) R_n^{-(N-1)}n^{-2} \quad \text{for $|x|\ge R_n$},
  \]
   so  $|u(x)|\le C'|x|^{-N+1}\omega(|x|) $ for $|x| \ge 2$. In particular $|u(x)|\le C\om(|x|)$.

  For $\la\in ((n+1)^{-2}\la_{n+1},n^{-2}\la_n]$ we estimate
  \begin{align*}
\la \nu_\ga(E_{\la,\gamma}[u]) 
\ge (n+1)^{-2} \la_{n+1} \nu_\ga(E_{n^{-2}\la_n,\gamma}[u] ) 
\ge \frac{\la_{n+1}}{4\la_n}\,  n^{-2}\la_{n} \nu_\ga( \cE_n)
  \end{align*}
where $\cE_n\coloneqq E_{n^{-2}\la_n,\ga}[u] \cap ([2R_n,3R_n] \times \R^{N-1})^2$. Moreover, for $(x,y) \in ([2R_n,3R_n] \times \R^{N-1})^2$, we have
 \[ 
 u(x)-u(y)= R_n^{1-N}n^{-2}\om(R_{n+1}) (f_{m(n)}(R_n^{-1} x)-f_{m(n)}(R_n^{-1} y) )
 \] 
 so
 \[
 |Q_{\ga}u(x,y)| > n^{-2}\la_n \Longleftrightarrow \frac{|f_{m(n)}(R_n^{-1} x)-f_{m(n)}(R_n^{-1}y)|}{|R_n^{-1} x - R_n^{-1} y|^{1+\ga}} > \frac{R_n^{N+\ga}}{\omega(R_{n+1})} \la_n = 1
 \]
 where the last equality follows from \eqref{eq:choice-of-global-parameters}.
 Hence  rescaling using \eqref{eq:mubeta-dil} yields
   \begin{align}
n^{-2} \la_{n} \nu_\ga(\cE_n) &= n^{-2} \la_n R_n^{\ga+N} \nu_\gamma(\{(x,y): x_1,y_1 \in [2,3], | Q_{\gamma} f_{m(n)} (x,y)|> 1 \}) \label{eq:rescale} \\ 
  &\ge c(N,\ga)  m(n) \omega(R_{n+1}) n^{-2} \notag
\end{align} 
  with $c(N,\ga)>0$, by  \eqref{Qbeta-fm-bd}.
  Thus we have shown
  \begin{align*}\inf_{\la\in ((n+1)^{-2}\la_{n+1},n^{-2}\la_n]} \la \nu_\ga(E_{\la,\gamma}[u])  
&\ge c(N,\ga) \frac{\la_{n+1}}{4\la_n} \omega(R_{n+1}) m(n)  n^{-2} \ge c(N,\ga) n\,
\end{align*} where for the last inequality we have used  our assumption \eqref{eq:choice-of-global-parameters} on $m(n)$. The assertion follows for $-1<\ga<0$.

   Finally consider the case $\gamma=-1$. We now choose 
   $v_m$ as in \eqref{eq:vmdef} and 
   \Be \label{eq:choice-of-global-parameters-minus1} 
   \text{ $R_n=2^{2n}$, \ $\la_n= R_n^{-(N-1)} \omega(R_{n+1})$, \ $m(n)\ge 4\frac{\la_n}{\la_{n+1}} \frac{n^3}{\omega(R_{n+1})}$.}
   \Ee
   In analogy to \eqref{eq:def-of-u} we now use
   \Be\label{eq:def-of-u-ga=1} 
   u(x) =\sum_{n=2}^\infty \frac{\om(R_{n+1})}{R_n^{N-1}n^2} v_{m(n)}(\frac{x}{R_n}) 
   \Ee
   Since $\om$ is bounded it is immediate  that $u\in \dot W^{1,1}(\Rset^N)$ and also that $|u(x)|\lc \om(|x|)$. We need to check that $\la \nu_{-1}(E_{\la,-1}[u])\to\infty$ as $\la\to 0^+$. 
   If $|x| \leq R_n(1-2^{m(n)})$ and $|y| \ge R_n(1+2^{m(n)})$, then 
   \[
   u(x)-u(y) \geq \frac{\omega(R_{n+1})}{R_n^{N-1} n^2} v_{m(n)}(\frac{x}{R_n})=2\frac{\om(R_{n+1})}{R_n^{N-1} n^2} = 2 n^{-2} \lambda_n > n^{-2} \lambda_n
   \]
   so $(x,y) \in E_{n^{-2}\la_n,-1}[u]$. Hence  we get 
   \begin{align*}
       n^{-2} \la_n \nu_{-1}(E_{n^{-2} \la_n,-1}[u]) &\ge n^{-2} \la_n \iint_{\substack{|x| \leq R_n(1-2^{m(n)}) \\ |y| \geq R_n(1+2^{m(n)})}} |x-y|^{-1-N} \dif x \dif y\\
       &\ge n^{-2} \la_n R_n^{N-1} \iint_{\substack{|x| \leq 1-2^{m(n)} \\|y| \geq 1+2^{m(n)}}} |x-y|^{-1-N} \dif x \dif y \\
       &\geq c_N m(n) \om(R_{n+1}) n^{-2} 
   \end{align*}
   (using \eqref{eq:choice-of-global-parameters-minus1} in the last  inequality). This together with our assumption on $m(n)$ imply that $\inf_{\la\in ((n+1)^{-2}\la_{n+1},n^{-2}\la_n]} \la \nu_{-1}(E_{\la,-1}[u]) \geq c_N n \to \infty$ when $n \to \infty$, as desired.
   \end{proof}
 
The next proposition is relevant for part (ii) of 
\cref{thm:counterexamples}.
\begin{proposition} \label{prop:comp-supp-example}
Suppose $-1 \leq \ga < 0$. Then there exists a compactly supported $u\in  W^{1,1}(\bbR^N)$ such that $u$ is $C^\infty$ for $x\neq 0$, 
\Be \label{eq:log-zero} 
|u(x)|\le \frac{C}{|x|^{N-1}[\log (2+|x|^{-1})]^2} 
\Ee 
and 
\Be
\label{eq:lambdatozero=infty-limit}  
 \lim_{\la\searrow 0} \la \nu_\gamma(E_{\la,\gamma}[u]) =\infty. 
\Ee 
If in addition $N \ge 2$ or $-1<\ga<0$ there exists  $u$ with the above properties and 
\Be
\label{eq:all-lambdainfty-limit}  
 \nu_\gamma(E_{\la,\gamma}[u]) =\infty  \text{ for all } \la > 0.
 \Ee
\end{proposition}

  \begin{proof}
Consider first the case $-1 < \ga < 0$.   We choose for $n\in \bbN$   
   \Be \label{eq:choice-of-local-parameters}   
   \text{ $R_n=2^{-2n}$,  $m(n)\ge 2^{2^n}$.}
   \Ee
 and with these choices of $R_n$ and $m(n)$ and $f_m$ as in \eqref{eq:fmdef} and \eqref{eq:umdef} we  define again  
 \[
 u(x) =\sum_{n=2}^\infty \frac{1}{n^2 R_n^{N-1}} f_{m(n)}(\frac{x}{R_n}).
 \]
  The sum converges in $W^{1,1}$ to a function supported in $[-4,4]^N$.
  We have $|u(x)|\le C 2^{2n(N-1)}n^{-2}$ for $0<x_1\le 2^{-2n}$;   moreover $|x'|\lc |x_1|$ on the support of $u$. This implies   $|u(x)|\le C'[|x|^{1-N}\log(1/|x|)]^{-2}$ for small $x$.
  Also, because of the choices of $R_n$ we see that $u$ is smooth away from $0$.
  
 Fix $\la > 0$. Since $\lim_{n\to\infty} R_n^{N+\ga} n^2=0$ we may choose $n_0$ such that 
  \Be \label{eq:large-n-fixed-la}
  \la R_n^{N+\gamma}n^2 \le 1, \quad \forall n\ge n_0.\Ee 
  Now $\nu_\ga(E_{\la,\ga}[u] ) \geq \nu_\ga(E_{\la,\ga}[u] \cap ([2R_n,3R_n] \times \R^{N-1})^2)$, and again $f_{m(n)}(R_n^{-1}\cdot)$ is supported in $\cR(n)=[R_n, 4R_n]\times [-4R_n, 4R_n]^{N-1}$. Hence by the same rescaling argument as in \eqref{eq:rescale}, we obtain
   \[
   \nu_\ga(E_{\la,\ga}[u] )
   \geq R_n^{N+\ga} \nu_\ga(\{(x,y) \colon x_1,y_1 \in [2,3], |Q_{\ga}f_{m(n)}(x,y)| > \la R_n^{N+\ga} n^2\}).
   \]
   If $n \ge n_0$ then this gives 
   \[
   \nu_\ga(E_{\la,\ga}[u] ) \geq R_n^{N+\ga} \nu_\ga(\{(x,y) \colon x_1,y_1 \in [2,3], |Q_{\ga}f_{m(n)}(x,y)| > 1\}) \geq c(N,\ga) m(n) R_n^{N+\gamma}
   \]
   by \eqref{Qbeta-fm-bd}.  Since $\lim_{n\to\infty} m(n) R_n^{N+\gamma}=\infty $ by  \eqref{eq:choice-of-local-parameters} we conclude $\nu_\gamma(E_{\la,\gamma} [u])=\infty$.

 For the case $\gamma=-1$ and $N \ge 2$, define $u$ as in  \eqref{eq:def-of-u-ga=1} but with the choice of the parameters $R_n$, $m(n)$ as in \eqref{eq:choice-of-local-parameters} to obtain a compactly supported  $u\in W^{1,1}$  satisfying \eqref{eq:log-zero}.   
 We now fix $\la>0$ and  note that when $N\ge 2$ we have $\la R_n^{N-1} n^2\to 0$ as $n \to \infty$. The  above calculation gives    $\nu_{-1}(E_{\la,-1}[u] ) \ge c(N) m(n) R_n^{N-1}$ provided that     $\la R_n^{N-1} n^2\le 1$ and thus the conclusion $\nu_{-1}(E_{\la, -1}[u]) =\infty$. 
    
Finally, clearly \eqref{eq:lambdatozero=infty-limit}  follows from 
\eqref{eq:all-lambdainfty-limit}, and the latter was proved if $-1<\ga<0$ or $N\ge 2$.  It remains to consider the case $N=1$, $\ga=-1$. We define $u$ as in the previous paragraph. The above calculation shows that $\nu_{-1} (E_{\la, -1}[u])\ge c m(n)$ provided that $\la < 1/n^2$ which establishes \eqref{eq:lambdatozero=infty-limit} in this last case.
 \end{proof} 
 
 The case $N=1$, $\gamma=-1$ plays a special role. The following lemma shows that the  conclusion \eqref{eq:all-lambdainfty-limit} in \cref{prop:comp-supp-example} fails in this case. 
\begin{lemma}\label{lem:uniform-cont}  
Let $u\in \dot W^{1,1}(\bbR) $ be compactly supported. Then $\nu_{-1} (E_{\la, -1}[u]) <\infty $ for all $\la>0$.
\end{lemma}
\begin{proof}
Let $u\in \dot W^{1,1}(\bbR)$ be compactly supported in $[-R,R]$. Then given any $\la \in (0,1)$, there exists  $\delta(\la)>0$ such that $\int_{I} |u'| \leq \lambda/2$ for every interval $I \subset \R$ with length $\leq \delta(\lambda)$. As a result, $u$ is uniformly continuous on $\R$, with $\sup_{x\in \R}|u(x+h)-u(x)|\le \la/2$  for $|h|\le \delta(\la)$. Thus 
\begin{align*}
\nu_{-1}(E_{\la,-1} [u])&= 2 \int_{-\infty}^\infty \int\limits_{\substack{h>0\\ |u(x+h)-u(x) |> \la } }\frac{\dif h}{h^2} \dif x
\\&
\leq \int_{-2R}^{2R} \int_{\delta(\la) }^\infty \frac{\dif h}{h^2} \dif x
+\int_{\bbR\setminus[-2R,2R]} \int_{|x|-R}^{|x|+R} \frac{\dif h}{h^2} \dif x 
\le  4R(\delta(\la))^{-1} +4. \qedhere
\end{align*}
\end{proof}

\subsection{Generic failure in $W^{1,1}$,  for the case $-1\le \ga<0$.}

\begin{proposition} \label{proposition:generic} 
Let $-1\le \gamma<0$, $N\ge 2$ or $-1<\ga<0$, $N\ge 1$. Let 
\Be\label{eq:cVdef} \cV= \big\{ f\in W^{1,1}(\Rset^N): \text{ $ \nu_\ga(E_{\la,\ga} [f])<\infty$ for some $\la>0$.}\}\Ee  Then $\cV$ is 
 of first category in $W^{1,1}(\Rset^N)$, 
in the sense of Baire.
\end{proposition}

Let 
\Be\label{eq:Omell} 
\begin{aligned}
 U_k&=\{(x,y)\in \bbR^{2N}: 2^{k-1}\le |x-y|\le 2^k\},
 \\ \Omega_\ell&={\textstyle \bigcup_{k=1-\ell}^\ell} U_k.
\end{aligned}
\Ee
For the proof of \cref{proposition:generic} we use  an elementary estimate for  the intersections $E_{\la,\ga}[u]\cap \Omega_\ell$.

\begin{lemma}\label{lem:elementary-LL} 
For all $\gamma\in \bbR$, $u\in W^{1,1}(\Rset^N)$, $\ell>0$ and $\Omega_\ell$ as in \eqref{eq:Omell},
\[ 
\sup_{\la>0} \la \nu_\ga(E_{\la,\ga}[u] \cap \Omega_\ell) \le C(N,\ga) \ell \|\nabla u\|_1
\]
\end{lemma}
\begin{proof} For $u\in C^1$ we use  the Lusin-Lipschitz inequality \eqref{eq:LL} 
to see that
\begin{multline*} 
\la \iint_{E_{\la,\ga} [u]\cap U_k} |x-y|^{\ga-N} \dif x\dif y \\ \le C(\ga) \la 2^{k \ga} \cL^N \{ x\in \bbR^N: M(|\nabla u|)(x) > c 2^{k \gamma}\la\}  \le C(N,\ga) \|\nabla u\|_1
\end{multline*}
by the Hardy-Littlewood maximal inequality.  Now sum in $1-\ell\le k\le \ell$. 
The extension to general $u\in W^{1,1}$ is obtained as in the limiting argument of \cref{sec:extending-to-Sobolev}.
\end{proof}

\begin{proof}[Proof of \cref{proposition:generic}]
Let, for $m\in \bbN$ and $j\in \bbZ$
\[
\cV(m,j)=\{u\in W^{1,1}(\Rset^N): \nu_{\gamma}(E_{\la,\ga}[u]) \le m \text{ for all $\la>2^j$}\}.
\]
Since $\la\mapsto \nu_\ga(E_{\la,\ga}[u])$ is decreasing we see that  $\cV$  is contained in $\bigcup_{m\ge 1}\bigcup_{j\in \bbZ} \cV(m,j)$. 
To show that $\cV$ is of first category in $W^{1,1} (\Rset^N)$, we need  to show that for every $m\in \bbN$, $j\in \bbZ$ the set $\cV(m,j)$ is nowhere dense.

We first show that $\cV(m,j)$ is  closed in $W^{1,1}(\Rset^N)$.
Let $u_n\in \cV(m,j)$  and $u\in  W^{1,1}(\Rset^N)$ such that $\lim_{n\to\infty}\|u-u_n\|_{W^{1,1}(\Rset^N)}= 0$. 
It suffices to show  that given $\eps>0$ we have $\nu_\ga(E_{\la,\ga} [u]) \le m+\eps$  for all $\la>2^j$.
  By the monotone convergence theorem, we have $\lim_{\ell\to \infty} \nu_\ga(E_{\la,\ga} [u]\cap \Omega_\ell) = \nu_\ga(E_{\la,\ga} [u])$, and it suffices to verify that 
 \Be \label{eq:closedness}
 \nu_\ga(E_{\la,\ga} [u] \cap \Omega_\ell) 
 \le m+\eps  \text{ for }\la>2^j, 
 \Ee
 for all $\ell \in \bbN$.
 Now let $\delta>0$ such that $(1-\delta)\la>2^j$. Then  
 \[
  \nu_\ga(E_{\la,\ga} [u] \cap \Omega_\ell) 
\le  \nu_{\gamma} (E_{ (1-\delta)\la,\ga}[u_n] \cap \Omega_\ell) + 
\nu_\gamma(E_{\delta \la ,\gamma}[u-u_n]\cap \Omega_\ell) 
\]
and using that $u_n\in \cV(m,j)$ together with $(1-\delta)\la>2^j$,  and \cref{lem:elementary-LL},  we see that for $\la>2^j$
\[
 \nu_\ga(E_{\la,\ga} [u] \cap \Omega_\ell) 
\le m + C(N,\ga)\ell \frac{1+\delta}{\delta 2^j}  \|\nabla(u_n-u)\|_1.
\]
Since $\delta>0$ was arbitrary and since 
$\|\nabla(u_n-u)\|_{L^1(\Rset^N)} \to 0$  by assumption we obtain
\eqref{eq:closedness}.

To show that  the closed set $\cV(m,j)$ is nowhere dense when $-1 \leq \ga < 0$ we need to verify  that for every $u\in\cV(m,j)$ and $\eps_1>0$ there exists $f\in W^{1,1}(\Rset^N)$ such that $\|f-u\|_{W^{1,1} (\Rset^N)}<\eps_1$ and $f\notin \cV(m,j)$. To see this we use \cref{prop:comp-supp-example} according to which there exists a compactly supported $W^{1,1}$ function $f_0$  for which $\nu_{\gamma}(E_{\la,\ga} [f_0])=\infty$ for all $\lambda >0$. It is then clear that  $f=u +  \frac {\eps_1} 2  \frac{f_0} {\|f_0\|_{W^{1,1}}}$ satisfies $\|f-u\|_{W^{1,1}}\le \eps_1/2$ and also, 
\[
\nu_{\ga}(E_{\la,\ga}[f]) \geq \nu_{\ga}(E_{2\la,\ga}[\tfrac{\varepsilon_1}{2} \tfrac{f_0}{\|f_0\|_{W^{1,1}}}]) - \nu_{\ga}(E_{\la,\ga}[u]) = \infty
\] 
for every $\la > 2^j$, for all $j\in \bbZ$. The  proposition is proved.
\end{proof} 

To include a result of generic failure of the limiting relation in the case $N=1$, $\gamma=-1$ we give
\begin{proposition} \label{proposition:generic-bis} 
Let $-1\le \gamma<0$. Let 
\[ 
\cW= \big\{ f\in W^{1,1} (\Rset): \text{ $ \limsup_{R\to 0} \sup_{\la>R} R \nu_\ga(E_{\la,\ga} [f])<\infty$}\}.
\]  
Then $\cW$ is of first category in $W^{1,1}$, in the sense of Baire.
\end{proposition}

\begin{proof} Clearly $\cW\subset \cV$ where $\cV$ is defined in \eqref{eq:cVdef}. 
We define 
\[
\cW(m,j)=\big\{u\in W^{1,1}(\Rset): \sup_{0<R\le 2^{-j}} \sup_{\la>R} R\nu_\ga(E_{\la,\gamma}[u])\le m\big\}
\] 
and note that
\Be\label{eq:cWincl} 
\cW \subset \cup_{j\ge 1}\cup_{m\ge 1} \cW(m,j)
\Ee
The arguments in the proof of \cref{proposition:generic} that was used to show that  the sets $\cV(m,j)$ are closed in $W^{1,1}(\Rset^N)$ also show that the sets $\cW(m,j)$ are closed in $W^{1,1}(\Rset)$. 

Let $u\in \cW(m,j)$, and let $\eps_1>0$. By \cref{prop:comp-supp-example} there is $f_0\in W^{1,1}(\Rset)$ such that $\lim_{\la\searrow 0} \la\nu_\ga(E_{\la,\ga}[f_0])=\infty$. We may normalize so that $\|f_0\|_{W^{1,1}(\Rset)} = 1$. Pick $R \in (0, 2^{-j}]$ so that $\la \nu_\ga(E_{\la,\ga} [f_0]) > 16m/\varepsilon_1$ for $\la \le 8R/\varepsilon_1$. Let $f=u+ (\eps_1/2) f_0$ so that $\|f-u\|_{W^{1,1}(\Rset)}\le \eps_1/2$. Moreover if $\la = 2R$, then $\la > R$ and 
\begin{align*}
&R \nu_\ga(E_{\la,\ga}[f]) 
\ge R \nu_\ga(E_{2\la,\ga}[\tfrac{\eps_1}2 f_0]) -  R  \nu_{\ga} (E_{\la,\ga} [u]) \\
&= \tfrac{\varepsilon_1}{8} \tfrac{8R}{\varepsilon_1} \nu_\ga(E_{8R/\varepsilon_1,\ga}[f_0]) - R \nu_{\ga} (E_{\la,\ga} [u]) > \tfrac{\varepsilon_1}{8} \tfrac{16m}{\varepsilon_1}-m = m
\end{align*}
and we see that $f\notin \cW(m,j)$.
Thus we have shown that $\cW(m,j)$ is nowhere dense in $W^{1,1}(\Rset)$.
By \eqref{eq:cWincl} the proof is concluded. 
\end{proof} 

\section{Perspectives and open problems}\label{sec:perspectives}

\subsection{\it Subspaces of $\dot W^{1,1}$ and $\dot{BV}$ and related spaces}
The failure of the upper bounds for $[Q_\gamma u]_{L^{1,\infty}(\R^{2N},\nu_\ga)}$ for $\gamma\in [-1,0)$ raises a number of interesting questions.
Consider the space $\dot{BV}(\ga)$ consisting of all $\dot {BV}$ functions satisfying 
\Be\label{eq:W11ga-qnorm} 
\|u\|_{\dot{BV}(\ga)} \coloneqq \|\nabla u\|_{\cM} +\sup_{\la>0} \la\nu_\gamma(E_{\la,\ga}[u])<\infty
\Ee
and the corresponding subspace $\dot W^{1,1}(\ga)$ of $\dot W^{1,1}$.  

\subsubsection*{Embeddings} We proved in this paper that for $\ga\notin[-1,0]$ we have $\dot {BV}(\gamma)=\dot{BV}$ and $\dot W^{1,1}(\ga)=\dot W^{1,1}$.  It is natural to ask how in  the range $-1\le \ga<0$ the proper subspaces  $\dot{BV}(\gamma)$  and $\dot W^{1,1}(\gamma)$  relate  to other families of function spaces, in particular to the  Hardy-Sobolev space $\dot F^{1}_{1,2}$, another subspace of $\dot W^{1,1}$. 

\subsubsection*{Triangle inequalities} The spaces $\dot W^{1,1}(\gamma)$ and $\dot {BV}(\gamma)$ are defined via $ L^{1,\infty}$ -quasi-norms, and the space  $L^{1,\infty}$ is not  normable (unlike $L^{p,\infty}$ for $1<p<\infty$ which is normable \cite{Hunt_1966}).
However \cref{theorem_weighted_pequal1} tells us that $\dot W^{1,1}(\gamma)$ and $\dot{BV}(\gamma)$ are normable for $\gamma\notin [-1,0]$. Are these spaces normable in the range  $\gamma\in [-1,0)$? 

\subsubsection*{Related quasi-norms} Consider for $0<s\le 1$ 
\[ 
\|u\|_{(p,s,\gamma)} = \Big [ \frac{u(x)-u(y)}{|x-y|^{\frac{\gamma}{p}+s}} \Big]_{L^{p, \infty}(\R^{2N}, \nu_\ga)}.
\]
It is an obvious consequence of \cref{theorem_weighted_pgr1} that for $s=1$ and fixed $p>1$, these expressions define  equivalent (semi/quasi)-norms on $C^\infty_c$ as $\ga$ varies over $\R \setminus\{0\}$. It would be interesting to find a more direct proof of this observation which does not involve the relation with $\dot W^{1,p}$. We note that the  equivalence for varying $\gamma$  breaks down for $0<s<1$. This result, and more  about the  spaces for which $\|u\|_{(p,s,\gamma)}<\infty$ with $0<s<1$,  such as their connection to Besov spaces and interpolation, can be found in  \cite{ssvy}.

\subsection{\it Other limit functionals}\label{sec:Otherliminfs}
Our results,  combined with the various developments presented in \citelist{\cite{BN2018}\cite{BN2020}\cite{Nguyen07}\cite{Nguyen11}}, suggest several possible directions of research.

Can one prove a generalization of $\eqref{eq:converse-p} $, $\eqref{eq:converse-1} $ where the supremum is replaced by the $\liminf_{\la\to\infty}$ when $\gamma>0$ and by a $\liminf_{\la\to 0^+}$ when $\gamma<0$. 
More precisely, for $1< p<\infty$ is there a positive constant $C(N,\gamma,p)$ such that for all $u\in L^1_\loc(\R^N)$
\begin{subequations}
\begin{align} 
\label{eq:liminfLppos}
\|\nabla u\|_{L^p}^p\le C(N,\gamma,p) 
\liminf_{\la\to\infty} \la^p\nu_\ga(E_{\la,\ga/p}[u])\text{ if } \gamma>0,
\\
\label{eq:liminfLpneg}
\|\nabla u\|_{L^p}^p \le C(N,\gamma,p) 
\liminf_{\la\searrow  0} \la^p \nu_\ga(E_{\la,\ga.p}[u])  \text{ if } \gamma<0,
 \end{align}
 \end{subequations} 
in the sense that $\|\nabla u\|_p=\infty$ if $u\in L^1_\loc\setminus \dot W^{1,p}$? 

For $p=1$ we can also ask: Is there a positive constant $C(N,\gamma)$ such that for all $u\in L^1_\loc(\R^N)$,
\begin{subequations}
\begin{align} 
\label{eq:liminfBVpos}
\|\nabla u\|_\cM\le C(N,\gamma) 
\liminf_{\la\to\infty} \la \nu_\ga(E_{\la,\ga}[u]) \text{ if } \gamma>0,
\\
\label{eq:liminfBVneg}
\|\nabla u\|_\cM\le C(N,\gamma) 
\liminf_{\la\searrow  0} \la \nu_\ga(E_{\la,\ga}[u]) \text{ if } \gamma<0,
 \end{align}
 \end{subequations}
 in the sense that $\|\nabla u\|_{\cM}=\infty$ if $u\in L^1_\loc\setminus \dot {BV}$?
 
\cref{theorem_weighted_reverse}  gives \eqref{eq:liminfLppos} and \eqref{eq:liminfLpneg} if we additionally assume $u \in \dot{W}^{1,p}(\R^N)$. It also gives \eqref{eq:liminfBVpos} and \eqref{eq:liminfBVneg} if we additionally assume that $u \in \dot{W}^{1,1}(\R^N)$. It would already be interesting to establish  \eqref{eq:liminfBVpos}, \eqref{eq:liminfBVneg} for all $\dot {BV}$ functions.

When $\ga = -1$, $p = 1$, \eqref{eq:liminfBVneg} holds for all $u \in L^1_\loc(\R^N)$ as established in  Nguyen \cite[Theorem 2]{Nguyen08} and Brezis--Nguyen \cite[Section 3.4]{BN2018}. For  $\ga=-p$, $1<p<\infty$ inequality  \eqref{eq:liminfLpneg} was proved in Bourgain--Nguyen \cite{Bourgain_Nguyen}. 
For   $\gamma=N$, Poliakovsky \cite{poliakovsky}  proved weaker versions of \eqref{eq:liminfLppos} and \eqref{eq:liminfBVpos}
where  the $\liminf$ is replaced by a $\limsup$. 

\subsection{\(\Gamma\)--convergence}\label{sec:Gammaconv}

This is a far-reaching generalization of the questions raised in \cref{sec:Otherliminfs}. 
For fixed $p \ge 1$ and $\ga \in \R \setminus \{0\}$ consider the functionals  
\Be\notag
 \Phi_{\lambda} [u]
 \defeq 
 \lambda^p \nu_\ga( E_{\lambda, \gamma/p} [u]), \quad \la \in (0,\infty)
\Ee
 defined for all  \(u\in L^1_\loc(\bbR^N)\).
It would be very interesting to study the  $\Gamma$-limit of \(\Phi_{\lambda}\) in $L^1_{\loc}(\R^N)$, in the sense of De Giorgi, as \(\lambda \to \infty\) when \(\gamma > 0\), resp.\ as \(\lambda \searrow 0\) when \(\gamma < 0\). 
More specifically, if $p>1$ define on $L^1_{\loc}(\R^N)$,
\[
\Phi_{*,c}[u]=
\begin{cases}  
c\|\nabla u\|_p^p &\text{ if } u\in \dot W^{1,p}(\bbR^N)
\\ \infty &\text{ otherwise},
\end{cases} 
\]
and for $p=1$ define
\[
\Phi_{*,c}[u]=
\begin{cases}  
c\|\nabla u\|_{\cM}  &\text{ if } u\in \dot {BV}(\bbR^N)
\\ \infty &\text{ otherwise}.
\end{cases} 
\]
A challenging question is whether there exists a constant $c=c(p,\ga,N)>0$ such that $\Phi_\lambda\to \Phi_{*,c}$  in the sense of $\Gamma$-convergence, meaning 
\begin{enumerate} 
\item whenever  $u_\la\to u$ in $L^1_\loc$ 
then $\liminf \Phi_\la[u_\la]\ge \Phi_{*,c}[u]$, and
\item for each $u \in L^1_{\loc}(\R^N)$ there exist $(v_\la)$ with $v_\la\in L^1_\loc(\bbR^N)$, $v_\la\to u$ in $L^1_\loc$ and $\limsup \Phi_\la[v_\la] \le \Phi_{*,c}[u]$.
\end{enumerate} 
This question is especially  meaningful in the  case $p=1$ where the pointwise limit behaves somewhat pathologically. 
 Indeed, recall that for $p=1$, $-1\le \gamma<0$ there is no universal upper bound for $\Phi_\la[u]$ in terms of $\|\nabla u\|_{L^1}$. 
Also when $p=1$ and $\ga \in \bbR\setminus [-1,0]$  the examples in \cref{sec:BV-limit} show that the pointwise limit in $\dot W^{1,1}$ and on $\dot{BV}\setminus \dot W^{1,1}$ may differ (by a multiplicative constant). A remarkable result of Nguyen  \citelist{\cite{Nguyen07}\cite{Nguyen11}} states that $\Phi_\la\to\Phi_{*,c} $ as $\la\to 0$, in the sense of  $\Gamma$--convergence, when $p\ge 1$, and \(\gamma = -p\) for some appropriate constant $c=c(p,N)$; see also  Brezis--Nguyen~\cite{BN2020} (note however that $\dot W^{1,p}$ and $\dot {BV}$ are replaced in these papers by $W^{1,p}$ and $BV$). 

\subsection{More general families of functionals}
Consider a monotone nondecreasing function \(\varphi: \intvr{0}{\infty} \to\intvr{0}{\infty}\) and set (inspired by \citelist{\cite{BN2018}\cite{BN2020}})
\[
 \Psi_{\lambda} [u] \defeq \lambda^p 
 \iint_{\Rset^N \times \Rset^N}
 \varphi \brk*{\frac{\abs{u (x) - u (y)}}{\lambda \abs{x- y}^{1 + \gamma/p}}} \abs{x - y}^{\gamma - N} \dif x \dif y.
\] 
The family \(\Phi_{\lambda}\) in \cref{sec:Gammaconv} corresponds to \(\varphi =\bbone_{(1,\infty)} \). It is an interesting generalization of the above  problems  to study  the limit of $\Psi_\la$ as $\la\searrow 0$ when \(\gamma < 0\) and the limit of $\Psi_\la$ as $\la\to \infty$ when \(\gamma > 0\),  both  in the sense of pointwise convergence or in the sense of $\Gamma$-convergence. 
A formal computation suggests that our \cref{theorem_weighted_reverse} should go  over modulo a factor \(\int_0^\infty \frac{\varphi (s)}{s^{p + 1}} \dif s\) (see \cite{BN2020}). We refer to \cite{BN2018} for a further  discussion of applications.

%\newpage
\setlength{\parskip}{1pt}

\begin{bibdiv}
\begin{biblist}

\bib{Bourgain_Brezis_Mironescu_2001}{article}{
  author={Bourgain, Jean},
  author={Brezis, Haim},
  author={Mironescu, Petru},
  title={Another look at Sobolev spaces},
  conference={
    title={Optimal control and partial differential equations},
  },
  book={
    publisher={IOS, Amsterdam},
  },
  date={2001},
  pages={439--455},
%   review={\MR{3586796}},
}

\bib{Bourgain_Nguyen}{article}{
    AUTHOR = {Bourgain, Jean},
    author = {Nguyen, Hoai-Minh},
     TITLE = {A new characterization of {S}obolev spaces},
   JOURNAL = {C. R. Math. Acad. Sci. Paris},
  FJOURNAL = {Comptes Rendus Math\'{e}matique. Acad\'{e}mie des Sciences. Paris},
    VOLUME = {343},
      YEAR = {2006},
    NUMBER = {2},
     PAGES = {75--80},
      ISSN = {1631-073X},
   MRCLASS = {46E35},
  MRNUMBER = {2242035},
MRREVIEWER = {Dorothee D. Haroske},
       DOI = {10.1016/j.crma.2006.05.021},
       URL = {https://doi-org.ezproxy.library.wisc.edu/10.1016/j.crma.2006.05.021},
}		
		
\bib{Brezis_2002}{article}{
   author={Brezis, Haim},
   title={How to recognize constant functions. A connection with Sobolev
   spaces},
   language={Russian},
   journal={Uspekhi Mat. Nauk},
   volume={57},
   date={2002},
   number={4(346)},
   pages={59--74},
   issn={0042-1316},
   translation={
      journal={Russian Math. Surveys},
      volume={57},
      date={2002},
      number={4},
      pages={693--708},
      issn={0036-0279},
   },
%    review={\MR{1942116}},
   doi={10.1070/RM2002v057n04ABEH000533},
}

\bib{Brezis_2011}{book}{
  author={Brezis, Haim},
  title={Functional analysis, Sobolev spaces and partial differential equations},
  series={Universitext},
  publisher={Springer, New York},
  date={2011},
  pages={xiv+599},
  isbn={978-0-387-70913-0},
%   review={\MR{2759829}},
}

\bib{BN2018}{article}{
   author={Brezis, Ha\"{\i}m},
   author={Nguyen, Hoai-Minh},
   title={Non-local functionals related to the total variation and
   connections with image processing},
   journal={Ann. PDE},
   volume={4},
   date={2018},
   number={1},
   pages={Art. 9, 77},
   issn={2524-5317},
%    review={\MR{3749763}},
   doi={10.1007/s40818-018-0044-1},
}

\bib{BN2020}{article}{
   author={Brezis, Ha\"{\i}m},
   author={Nguyen, Hoai-Minh},
   title={Non-local, non-convex functionals converging to Sobolev norms},
   journal={Nonlinear Anal.},
   volume={191},
   date={2020},
   pages={111626, 9},
   issn={0362-546X},
%    review={\MR{4011115}},
   doi={10.1016/j.na.2019.111626},
}

\bib{Brezis_VanSchaftingen_Yung_2021}{article}{
  title={A surprising formula for Sobolev norms},
  author={Brezis, Ha\"{\i}m},
  author={Van Schaftingen, Jean},
  author={Yung, Po-Lam},
  doi={10.1073/pnas.2025254118},
  journal={Proc. Natl. Acad. Sci. USA},
  year={2021},
  volume={118},
  number={8},
  pages={e2025254118},
}

\bib{Brezis_VanSchaftingen_Yung_2021Lorentz}{article}{
     TITLE = {Going to {L}orentz when fractional {S}obolev, {G}agliardo and
              {N}irenberg estimates fail},
   author={Brezis, Ha\"{\i}m},
  author={Van Schaftingen, Jean},
  author={Yung, Po-Lam},            
   JOURNAL = {Calc. Var. Partial Differential Equations},
    VOLUME = {60},
      YEAR = {2021},
    NUMBER = {4},
     PAGES = {Paper No. 129, 12},
      ISSN = {0944-2669},
       DOI = {10.1007/s00526-021-02001-w},
       URL = {https://doi-org.ezproxy.library.wisc.edu/10.1007/s00526-021-02001-w},
}

\bib{Davila}{article}{
   author={D\'{a}vila, J.},
   title={On an open question about functions of bounded variation},
   journal={Calc. Var. Partial Differential Equations},
   volume={15},
   date={2002},
   number={4},
   pages={519--527},
   issn={0944-2669},
   review={\MR{1942130}},
   doi={10.1007/s005260100135},
}

\bib{dominguez-milman}{article}{
 author = {Dom\'inguez, \'O.},
 author = {Milman, M.},
  title  = {New Brezis-Van Schaftingen-Yung-Sobolev type inequalities connected with maximal inequalities and one parameter families of operators},  
  year ={2020},
  journal= {Preprint, {arxiv:2010.15873}},
}


\bib{Federer}{book}{
   author={Federer, Herbert},
   title={Geometric measure theory},
   series={Die Grundlehren der mathematischen Wissenschaften, Band 153},
   publisher={Springer-Verlag New York Inc., New York},
   date={1969},
   pages={xiv+676},
%    review={\MR{0257325}},
}
		
\bib{hajlasz-kalamajska}{article}{
    AUTHOR = {Haj\l asz, P.},
    Author = {Ka\l amajska, A.},
     TITLE = {Polynomial asymptotics and approximation of {S}obolev  functions},
   JOURNAL = {Studia Math.},
  FJOURNAL = {Studia Mathematica},
    VOLUME = {113},
      YEAR = {1995},
    NUMBER = {1},
     PAGES = {55--64},
      ISSN = {0039-3223},
   MRCLASS = {46E35 (41A10)},
  MRNUMBER = {1315521},
MRREVIEWER = {J. Horv\'{a}th},
}

\bib{Hunt_1966}{article}{
  author={Hunt, Richard A.},
  title={On \(L(p, q)\) spaces},
  journal={Enseign. Math. (2)},
  volume={12},
  date={1966},
  pages={249--276},
  issn={0013-8584},
%   review={\MR{223874}},
}

\bib{mattila}{book}{,
    AUTHOR = {Mattila, Pertti},
     TITLE = {Fourier analysis and {H}ausdorff dimension},
    SERIES = {Cambridge Studies in Advanced Mathematics},
    VOLUME = {150},
 PUBLISHER = {Cambridge University Press, Cambridge},
      YEAR = {2015},
     PAGES = {xiv+440},
      ISBN = {978-1-107-10735-9},
   MRCLASS = {28-02 (28A15 28A78 28A80 42B10 60J65)},
  MRNUMBER = {3617376},
MRREVIEWER = {Benjamin Steinhurst},
       DOI = {10.1017/CBO9781316227619},
       URL = {https://doi-org.ezproxy.library.wisc.edu/10.1017/CBO9781316227619},
}


\bib{Nguyen06}{article}{
   author={Nguyen, Hoai-Minh},
   title={Some new characterizations of Sobolev spaces},
   journal={J. Funct. Anal.},
   volume={237},
   date={2006},
   number={2},
   pages={689--720},
   issn={0022-1236},
%    review={\MR{2230356}},
   doi={10.1016/j.jfa.2006.04.001},
}

\bib{Nguyen07}{article}{
   author={Nguyen, Hoai-Minh},
   title={$\Gamma$-convergence and Sobolev norms},
%    language={English, with English and French summaries},
   journal={C. R. Math. Acad. Sci. Paris},
   volume={345},
   date={2007},
   number={12},
   pages={679--684},
   issn={1631-073X},
%    review={\MR{2376638}},
   doi={10.1016/j.crma.2007.11.005},
}

\bib{Nguyen08}{article}{
   author={Nguyen, Hoai-Minh},
   title={Further characterizations of Sobolev spaces},
   journal={J. Eur. Math. Soc. (JEMS)},
   volume={10},
   date={2008},
   number={1},
   pages={191--229},
   issn={1435-9855},
   review={\MR{2349901}},
}

\bib{Nguyen11}{article}{
   author={Nguyen, Hoai-Minh},
   title={$\Gamma$-convergence, Sobolev norms, and BV functions},
   journal={Duke Math. J.},
   volume={157},
   date={2011},
   number={3},
   pages={495--533},
   issn={0012-7094},
%    review={\MR{2785828}},
   doi={10.1215/00127094-1272921},
}

\bib{poliakovsky}{article}{
 author = {A. Poliakovsky},
  title  = {Some remarks on a formula for Sobolev norms due to Brezis, Van Schaftingen and Yung},
  year ={2021},
  journal= {Preprint, {arxiv:2102.00557}},
}

\bib{ssvy}{article} {
author = {Seeger, Andreas},
author = {Street, Brian},
author = {Van Schaftingen, Jean},
author = {Yung, Po-Lam},
title = {On spaces of Besov-Sobolev type},
journal = {Preprint},
}

\bib{stein-diff}{book}{
    AUTHOR = {Stein, Elias M.},
     TITLE = {Singular integrals and differentiability properties of
              functions},
    SERIES = {Princeton Mathematical Series, No. 30},
 PUBLISHER = {Princeton University Press, Princeton, N.J.},
      YEAR = {1970},
     PAGES = {xiv+290},
   MRCLASS = {46.38 (26.00)},
  MRNUMBER = {0290095},
MRREVIEWER = {R. E. Edwards},
}
		
\bib{stein-weiss}{book}{
    author = {Stein, Elias M.},
    author= {Weiss, Guido},
     title = {Introduction to {F}ourier analysis on {E}uclidean spaces},
      note = {Princeton Mathematical Series, No. 32},
 publisher = {Princeton University Press, Princeton, N.J.},
      year = {1971},
     pages = {x+297},
   mrclass = {42A92 (31B99 32A99 46F99 47G05)},
  mrnumber = {0304972},
mrreviewer = {Edwin Hewitt},
}

\end{biblist}
\end{bibdiv}
\end{document}